\documentclass[a4paper,12pt]{amsart}
\usepackage{amssymb,amstext,amsmath,amscd,amsthm,amsfonts,enumerate,latexsym, comment}
\usepackage{color}
\usepackage{graphicx}

\usepackage[all]{xy}
\usepackage{hyperref}
\usepackage{needspace,booktabs,longtable}
\hypersetup{hidelinks,pdfborder={0 0 0},pdftitle={Divisorial ideals and v-class semigroups of strictly closed domains},pdfauthor={Ryotaro Isobe},pdfkeywords={strict closure, Arf ring, divisorial ideal, v-class semigroup, Boole v-regularity}}

\theoremstyle{definition}

\theoremstyle{plain}

\theoremstyle{plain}
\newtheorem{thm}{Theorem}[section]

\newtheorem*{thm*}{Theorem}
\newtheorem*{cor*}{Corollary}
\newtheorem*{thma}{Theorem A}
\newtheorem*{thmb}{Theorem B}

\newtheorem{prop}[thm]{Proposition}

\newtheorem{lem}[thm]{Lemma}
\newtheorem{cor}[thm]{Corollary}

\newtheorem{claim}{Claim}
\newtheorem*{claim*}{Claim}

\theoremstyle{definition}
\newtheorem{dfn}[thm]{Definition}

\newtheorem{ex}[thm]{Example}

\newtheorem{rem}[thm]{Remark}

\theoremstyle{remark}

\numberwithin{equation}{thm}

\def\Min{\operatorname{Min}}

\def\Im{\operatorname{Im}}

\def\Ker{\operatorname{Ker}}

\def\Hom{\operatorname{Hom}}

\def\Max{\operatorname{Max}}
\def\Spec{\operatorname{Spec}}

\def\End{\mathrm{End}}

\def\m{\mathfrak m}

\def\H{\mathrm{H}}

\newcommand{\rmA}{\mathrm{A}}

\newcommand{\calS}{\mathcal{S}}

\newcommand{\fka}{\mathfrak{a}}

\newcommand{\fkc}{\mathfrak{c}}

\newcommand{\fkm}{\mathfrak{m}}

\newcommand{\fkp}{\mathfrak{p}}

\newcommand{\bsF}{\boldsymbol{F}}

\renewcommand{\bar}{\overline}

\newcommand{\benu}{\begin{enumerate}}

\newcommand{\enenu}{\end{enumerate}}

\newcommand{\mapright}[1]{%
\smash{\mathop{%
\hbox to 1cm{\rightarrowfill}}\limits^{#1}}}

\newcommand{\mapleft}[1]{%
\smash{\mathop{%
\hbox to 1cm{\leftarrowfill}}\limits_{#1}}}

\def\depth{\operatorname{depth}}

\def\Ass{\operatorname{Ass}}

\def\height{\mathrm{ht}}
\def\Spec{\operatorname{Spec}}

\def\B{{\mathcal B}}

\def\F{{\mathcal F}}

\DeclareMathOperator{\Inv}{Inv}
\DeclareMathOperator{\Prin}{Prin}
\DeclareMathOperator{\Cl}{Cl}
\DeclareMathOperator{\tr}{tr}

\title[Divisorial ideals and $v$-class semigroups]{Divisorial ideals and $v$-class semigroups
of strictly closed domains}
\author[R. Isobe]{Ryotaro Isobe}
\address{Faculty of Education, Chiba University,
1-33 Yayoi-cho, Inage-ku, Chiba-shi, Chiba-ken, 263-8522, Japan}
\email{ryotaro.isobe@faculty.gs.chiba-u.jp}
\date{}

\thanks{The author was supported by JSPS KAKENHI Grant Number JP24K16910.}

\begin{document}
\setlength{\baselineskip}{15.2pt}

\begin{abstract}
We study divisorial ideals and the $v$-class semigroup of a Noetherian domain satisfying Serre's condition $(S_2)$.
We prove that strict closedness is equivalent to Clifford $v$-regularity and describe the structure of the associated $v$-class semigroup.
We use these results to characterize finite generation of the class semigroup and Boole $v$-regularity.
We also construct prime ideal decompositions of integrally closed divisorial ideals and clarify the relationship between these decompositions and local Boole $v$-regularity.
Finally, we study polynomial extensions and use them to construct
Boole $v$-regular and locally Boole $v$-regular domains in arbitrary dimension.
\end{abstract}

\maketitle

\section{Introduction}\label{intro}
In this paper, we study divisorial ideals of strictly closed domains and the structure of their $v$-class semigroups.
For a commutative ring $R$, the strict closure $R^*$ is the intermediate ring consisting of the elements $\alpha$ of its integral closure $\bar R$ in the total ring of fractions such that $\alpha\otimes1=1\otimes\alpha$ in $\bar R\otimes_R\bar R$,
and we say that $R$ is strictly closed if $R=R^*$.
This notion was introduced by Lipman \cite{L}. 
In the same paper, Lipman introduced the Arf property for Cohen-Macaulay semi-local rings whose localizations at all maximal ideals have dimension one, and characterized it in terms of the stability of integrally closed ideals (\cite[Theorem 2.2]{L}).
In this setting, the Arf property is equivalent to strict closedness (\cite[Proposition 4.5, Theorem 4.6]{L}, \cite[Theorem 4.4]{C}).
Moreover, a Noetherian domain satisfying Serre's condition $(S_2)$ is strictly closed if and only if its localizations at height-one prime ideals are Arf (\cite[Proposition 4.2, Theorem 4.4]{C}).
We use this local criterion to relate the theory of divisor class groups of integrally closed domains to the ideal theory of one-dimensional Arf rings.
For related work on Arf rings and strict closure, see also \cite{EGI,I1,I2,IK}.

Throughout this section, let $R$ be a Noetherian domain satisfying Serre's condition $(S_2)$, and let $K$ be its field of fractions.
For nonzero fractional ideals $X,Y$, set $X:Y=\{a\in K\mid aY\subseteq X\}$ and $X^{v_R}=R:(R:X)$.
We say that $X$ is \emph{divisorial} if $X=X^{v_R}$.
The isomorphism classes of nonzero divisorial fractional ideals form a commutative monoid under the product $[X][Y]=[(XY)^{v_R}]$.
This monoid is called the \emph{$v$-class semigroup} of $R$ and is denoted by $\calS^v(R)$.
Its group of units is called the \emph{$v$-class group} and is denoted by $\Cl^v(R)$.
It consists of the isomorphism classes of divisorial fractional ideals satisfying $(X(R:X))^{v_R}=R$, which are called \emph{$v$-invertible} fractional ideals.
In this setting, $\calS^v(R)=\Cl^v(R)$ if and only if $R$ is integrally closed.

A commutative semigroup $S$ is called \emph{Clifford} if $S$ is a disjoint union of subgroups.
It is called \emph{Boolean} if every element of $S$ is idempotent.
Every Boolean semigroup is Clifford.
A domain $R$ is said to be \emph{Clifford $v$-regular} (resp.\ \emph{Boole $v$-regular}) if $\calS^v(R)$ is Clifford (resp.\ Boolean).
Clifford and Boole regularity of class semigroups were studied by Kabbaj and Mimouni \cite{KM,KMN}, and their relationship with stability with respect to star operations was investigated by Gabelli and Picozza \cite{GP2,GP}.

Our first objective is to characterize strict closedness in terms of
the class semigroup and, in the strictly closed case, to describe
this semigroup as a disjoint union of groups.
Let $\B^v(R)$ denote the set of finite birational extensions that are divisorial as $R$-modules.
The first main result of this paper is as follows.

\begin{thma}[Theorem \ref{5.5}]
Suppose that $R$ is a Noetherian domain satisfying Serre's condition $(S_2)$. Then, the following conditions are equivalent.
\benu[$(1)$]
\item $R=R^*$.
\item $R$ is Clifford $v$-regular.
\item For every nonzero divisorial ideal $I\subseteq R$, $I$ is $v$-invertible over $A$, where $A=I:I$.
\enenu
When these conditions hold, setting $\calS_A^v(R)=\{[I]\in\calS^v(R)\mid I:I=A\}$ for each $A\in\B^v(R)$, we have
\[
 \calS^v(R)=\bigsqcup_{A\in\B^v(R)}\calS_A^v(R) \quad \text{and}
 \quad  \calS_A^v(R)\cong\Cl^v(A),
\]
where the latter is a natural isomorphism of abelian groups, and the identity of the corresponding constituent group is $[A]$.
\end{thma}

The relationship between Clifford $t$-regularity and invertibility
over endomorphism rings has been studied in \cite{KMN,GP2,GP}.
In the present $(S_2)$ setting, Theorem A connects this theory with
strict closedness and describes the constituent groups in terms of
the $v$-class groups of the rings $A\in\B^v(R)$.
Each such $A$ is again a strictly closed Noetherian domain satisfying $(S_2)$
(Proposition~\ref{2.10} and Theorem~\ref{3.12}).
We use this description to study finite generation and the relationship
between $v$-multiplication and ordinary ideal multiplication.
More precisely, under the assumption $R=R^*$, the monoid $\calS^v(R)$ is finitely generated if and only if $\bar R$ is a finitely generated $R$-module and $\Cl^v(R)$ is a finitely generated abelian group (Theorem~\ref{5.14}).
Also, we prove that $R$ is Boole $v$-regular if and only if $R=R^*$ and $\Cl^v(R)=0$.
These conditions are equivalent to the requirement that, for every nonzero divisorial ideal $I\subseteq R$, there exists $a\in I$ such that $I^2=aI$ (Theorem~\ref{5.9}).

Our second objective is to construct a prime ideal decomposition of integrally closed divisorial ideals and to determine when this decomposition holds under ordinary ideal multiplication.
In \cite{I1}, the author decomposed integrally closed ideals of one-dimensional Arf rings as ordinary products of prime ideals appearing in an ascending chain of rings.
Here we extend this construction to arbitrary dimension by using $v$-multiplication.
More precisely, assuming $R=R^*$, we associate a finite ascending chain of finite birational extensions with each nonzero integrally closed divisorial ideal $I\subseteq R$.
If $P(I)$ denotes the ordinary product of certain height-one prime ideals of these rings, then the equation 
 $I=(P(I))^{v_R}$ holds
(Proposition~\ref{6.3}).
The construction of $P(I)$ is given in Section~\ref{sec:decomposition}.
The following theorem gives conditions under which the $v$-closure can be omitted from this decomposition.

\begin{thmb}[Theorem~\ref{6.5}]
Suppose that $R$ is a Noetherian domain satisfying $(S_2)$, and $R=R^*$. Then, the following conditions are equivalent.
\benu[$(1)$]
\item For every $P\in\Spec R$, $R_P$ is Boole $v$-regular.
\item The ordinary product $XY$ of any two nonzero divisorial fractional ideals $X,Y$ is divisorial.
\item Every divisorial fractional ideal that is $v$-invertible over $R$ is invertible over $R$.
\enenu
When these conditions hold, the construction above gives $I=P(I)$ for every nonzero integrally closed divisorial ideal $I\subseteq R$.
In addition, if $\bar R$ is a finitely generated $R$-module, then these conditions are also equivalent to the requirement that $I=P(I)$ for every nonzero integrally closed divisorial ideal $I\subseteq R$.
\end{thmb}

When $R$ is integrally closed, Theorem B corresponds to the characterization of $\pi$-domains among Krull domains (\cite[Theorem 1]{A}).
In this case, when the equivalent conditions of the theorem hold, the ordinary product decomposition above involves height-one prime ideals of $R$ itself and agrees with the prime ideal decomposition in a $\pi$-domain.


The rest of this paper is organized as follows.
In Section~\ref{sec:arf}, we introduce the notation and recall results on the Arf property and strict closedness.
In Section~\ref{sec2}, we collect basic properties of divisorial ideals. In Section~\ref{sec:overrings}, we show that $(S_2)$ and strict closedness pass to the finite birational extensions under consideration (Proposition \ref{2.10} and Theorem \ref{3.12}).
In Section~\ref{sec:semigroups}, we prove Theorem A (Theorem \ref{5.5}) and study finite generation of the class semigroup and Boole $v$-regularity.
In Section~\ref{sec:decomposition}, we construct decompositions of integrally closed divisorial ideals and prove Theorem B (Theorem \ref{6.5}).
In Section~\ref{sec:polynomial}, we study polynomial extensions.
Clifford $v$-regularity is equivalent for $R$ and $R[X]$, whereas $R[X]$ is Boole $v$-regular if and only if $R$ is a UFD.
Furthermore, $R[X]$ is locally Boole $v$-regular if and only if $R$ is locally a UFD (Theorems~\ref{7.1}, \ref{7.4}, \ref{7.5}).
On the other hand, suitable localizations of polynomial extensions yield, for every positive integer $d$, a non-integrally closed Boole $v$-regular domain of dimension $d$.
We also construct non-integrally closed domains of dimension $d$ that are locally, but not globally, Boole $v$-regular (Examples~\ref{ex:boole-anydim}, \ref{ex:locally-boole-anydim}).

\section{The Arf property and strict closedness}\label{sec:arf}
In this section, we recall results on the Arf property and strict closedness.
For a Noetherian ring $R$, let $K=Q(R)$ be its total ring of fractions, and let $\bar R$ be its integral closure in $K$.
We denote by $\bsF(R)$ the set of finitely generated $R$-submodules of $K$ containing a non-zerodivisor, and call these modules fractional ideals.
For a domain, this is the set of nonzero fractional ideals.
Set
\[
 \B(R)=\{A\in\bsF(R)\mid R\subseteq A\subseteq K,\ A\text{ is a subring of $K$}\},
 \quad \F(R)=\{I\in\bsF(R)\mid I\subseteq R\}.
\]
The elements of $\B(R)$ are finite birational extensions.
For $X,Y\in\bsF(R)$, set $X:Y=\{a\in K\mid aY\subseteq X\}$ and $X^{v_R}=R:(R:X)$.
We say that $X$ is divisorial over $R$ if $X=X^{v_R}$.
The sets of divisorial elements of $\bsF(R),\B(R),\F(R)$ are denoted by $\bsF^v(R),\B^v(R),\F^v(R)$, respectively.
Also, $\bar I$ denotes the integral closure of an ideal $I$, and we set
\[
 \H_1(R)=\{P\in\Spec R\mid \height_R P=1\},\quad
 \H_1^{\rm reg}(R)=\{P\in\H_1(R)\mid R_P\text{ is a DVR}\}.
\]

Until Theorem~\ref{3.8}, unless otherwise specified, we assume that $R$ is a one-dimensional Cohen-Macaulay semi-local ring such that $\dim R_M=1$ for every $M\in\Max R$.
We begin with the definition of an Arf ring.

\begin{dfn}[\cite{L}]\label{3.1}
We call $R$ an {\it Arf ring} if the following conditions hold.

\begin{enumerate}[$(1)$]
\item
For every integrally closed ideal $I\in \F(R)$, there exists $a\in I$ such that $I=\bar{aR}$; that is, $I^{n+1}=aI^n$ for some $n\ge 0$.
\item
For $x, y, z \in R$, if $x$ is a non-zerodivisor of $R$ and $\frac{y}{x}, \frac{z}{x}\in \overline{R}$, then $\frac{yz}{x}\in R$.
\end{enumerate}

\end{dfn}

The following characterization of the Arf property is known.

\begin{thm}[{\cite[Theorem 2.2]{L}}]\label{3.2}
The following conditions are equivalent.
\benu[$(1)$]
\item
$R$ is Arf.
\item
For every integrally closed ideal $I\in \F(R)$, there exists $a\in I$ such that $I^2=aI$.

\enenu
\end{thm}

For a finitely generated module $M$, set
\[
\tr_R(M)=\sum_{f\in\Hom_R(M,R)}\operatorname{Im}f.
\]
An ideal $I$ is called a {\it trace ideal} if $I=\tr_R(M)$ for some finitely generated module $M$.
For a fractional ideal $X$, we have $\tr_R(X)=X(R:X)$.
An ideal $I\in\F(R)$ is a trace ideal if and only if $I=\tr_R(I)$, which is equivalent to $I:I=R:I$ (\cite[Lemma 2.1, Corollary 2.2]{IK}).

\begin{thm}[{\cite[Theorem 3.5]{IK}}]\label{3.3}
Suppose that $R$ is Arf. For $I\in \F(R)$, the following conditions are equivalent.
\benu[$(1)$]
\item
$I\in \F^v(R)$.
\item
There exists a trace ideal $J\in \F(R)$ such that $I\cong J$.
\item
There exists an integrally closed ideal $J\in \F(R)$ such that $I\cong J$.
\enenu

When this is the case, $I\cong \tr_R(I) \cong R:I$, and $\tr_R(I)$ is integrally closed.
\end{thm}

By the preceding theorem, we obtain the following.

\begin{thm}\label{3.4}
The following conditions are equivalent.
\benu[$(1)$]
\item
$R$ is Arf.
\item
For every $I\in \F^v(R)$, there exists $a\in I$ such that $I^2=aI$.
\item
For every $I\in \F^v(R)$, we have $I\cong \tr_R(I)$.
\enenu

\end{thm}

\begin{proof}
$(1)\Rightarrow (3):$ This follows from Theorem~\ref{3.3}.

$(3)\Rightarrow (2):$
Set $A=I:I$. Since $I$ is divisorial, we have $R:\tr_R(I)=R:I(R:I)=I^{v_R}:I=A$.
By assumption, $\tr_R(I)$ is also divisorial, so $\tr_R(I)=R:(R:\tr_R(I))=R:A$.
Choose a non-zerodivisor $d\in R$ such that $dA\subseteq R$, and set $J=dA$.
Since $J\in \F^v(R)$, applying assumption $(3)$ to $J$ gives
$$A\cong J\cong \tr_R(J)=\tr_R(A)=A(R:A)=R:A=\tr_R(I)\cong I.$$
Thus, there exists $a\in K^{\times}$ such that $I=aA$.
Since $1\in A$, we have $a\in I$, and $I^2=a^2A^2=a(aA)=aI$.

$(2)\Rightarrow (1):$
First, we note that $B^{v_R}\subseteq \bar{R}$ for every $B\in\B(R)$. In fact, set $\mathfrak c=R:B$. Since $\fkc=\fkc B=\tr_R(B)$, we have
$B^{v_R}=R:\mathfrak c=\mathfrak c:\mathfrak c\subseteq\bar R$.

It suffices to show that every integrally closed ideal $I\in \F(R)$ is divisorial.
By \cite[Proposition 1.1(i), (ii)]{L}, there exist $A\in \B(R)$ and a non-zerodivisor $\alpha\in A$ such that $IA=\alpha A$, and we can write $I=\bar{I}=\bar{I\bar{R}}\cap R=\bar{\alpha\bar{R}}\cap R=\alpha \bar{R}\cap R$.
Hence, we have
$$R:(R:I)\subseteq (R:(R:IA))\cap R=(R:(R:\alpha A))\cap R=\alpha A^{v_R}\cap R\subseteq \alpha \bar{R}\cap R=I,$$
so $I$ is divisorial.
\end{proof}

\begin{rem}
Condition $(2)$ of Theorem~\ref{3.4} is equivalent to the requirement that every $I\in \F^v(R)$ be a principal ideal of its endomorphism ring $I:I$.
When $R$ satisfies this condition, $R$ is said to be {\it strongly $v$-stable} (\cite[\S1]{GP}).
For a related characterization of the Arf property in terms of reflexive fractional ideals, see also \cite[Theorem 5.7]{CDS}.
The latter result assumes a one-dimensional analytically unramified local ring with infinite residue field; it does not directly replace the assertion for semi-local rings used here.

\end{rem}

The following result concerns the Arf property of finite birational extensions.

\begin{thm}[{\cite[Corollary 2.5]{I1}}]\label{3.5}
Suppose that $R$ is Arf, and let $I\in \F(R)$ be integrally closed. Then $I:I$ is also Arf.
\end{thm}

By Theorems~\ref{3.3} and \ref{3.5}, we obtain the following.

\begin{cor}\label{3.6}
If $R$ is Arf, then every $A\in \B^v(R)$ is also Arf.
\end{cor}

\begin{proof}
Let $A\in \B^v(R)$. Choose a non-zerodivisor $a\in R$ such that $aA\subseteq R$, and set $I=aA$.
Since $I\cong A$, we have $I\in \F^v(R)$.
By Theorem~\ref{3.3}, there exists an integrally closed ideal $J\in \F(R)$ such that $I\cong J$.
Since $J:J=I:I=A$, Theorem~\ref{3.5} shows that $A$ is Arf.
\end{proof}

\begin{ex}\label{ex:arf-chain}
Let $k$ be a field and $V=k[[t]]$ be the formal power series ring over $k$. Set $R_0=V$ and $R_n=k[[t^2,t^{2n+1}]]\subseteq V$ $(n\ge 1)$.
Then $R_n$ is Arf and
\[
 \B^v(R_n)=\B(R_n)=\{R_i \mid 0\le i\le n\}.
\]
\end{ex}

We next recall the definition of strict closure.

\begin{dfn}[\cite{L}]\label{3.7}
For an extension $S/R$ of commutative rings, set
$$
R_S^{*}=\left\{ \alpha \in S \mid \alpha \otimes 1 = 1 \otimes \alpha \text{ in } S \otimes_RS\right\}
$$
and call this ring the {\it strict closure} of $R$ in $S$.
When $S=\overline{R}$, we write simply $R^*$ and call it the strict closure of $R$.
We say that $R$ is \emph{strictly closed} if $R=R^*$.
\end{dfn}

In general, $R_S^*$ is the kernel of the $R$-linear map $\delta:S\longrightarrow S\otimes_RS$ defined by $\alpha\longmapsto\alpha\otimes1-1\otimes\alpha$.
The exactness of localization therefore gives
\[
 W^{-1}(R_S^*)=(W^{-1}R)^*_{W^{-1}S}
\]
for every multiplicatively closed subset $W\subseteq R$.
In particular, if $R$ is Noetherian and satisfies $(S_1)$, then $W^{-1}\overline R=\overline{W^{-1}R}$, and hence $(W^{-1}R)^*=W^{-1}(R^*)$ (\cite[Lemma 4.1]{C}).
Thus, in this setting, strict closedness is preserved under arbitrary localization.

The following two theorems describe the relationship between strict closedness and the Arf property.

\begin{thm}[{\cite[Proposition 4.5, Theorem 4.6]{L}}, {\cite[Theorem 4.4]{C}}]\label{3.8}
Under the standing assumptions on Cohen-Macaulay semi-local rings of dimension one in this section, $R=R^*$ if and only if $R$ is Arf.
\end{thm}

\begin{thm}[{\cite[Proposition 4.2, Theorem 4.4]{C}}]\label{3.9}
Let $R$ be a Noetherian ring satisfying $(S_2)$. Then the following conditions are equivalent.
\benu[$(1)$]
\item
$R=R^*$.
\item
For every $P\in \H_1(R)$, $R_P$ is Arf.

\enenu

\end{thm}

\section{Basic properties of divisorial ideals}\label{sec2}

In this section, we collect the basic properties of divisorial ideals needed in the subsequent sections.
For the general theory of the $v$-operation, see, for example, \cite{G}; for localization, associated prime ideals, and primary decomposition of divisorial ideals in Mori domains, see \cite{HLV}.
We state and prove these results in the present setting.

Throughout the remainder of this paper, unless otherwise specified,
we assume that $R$ is a Noetherian domain satisfying Serre's condition $(S_2)$ and is not a field.
Colon operations on finitely generated fractional ideals commute with localization, and hence $(X^{v_R})_P=(X_P)^{v_{R_P}}$.

\begin{lem}\label{2.1}
We have $\displaystyle R=\bigcap_{P\in \H_1(R)} R_P$.
\end{lem}

\begin{proof}
Set $A=\bigcap_{P\in \H_1(R)} R_P$. The inclusion $R\subseteq A$ is clear.
Suppose that $R\subsetneq A$, and choose $x\in A\setminus R$.
We write $x=\frac{b}{a}$ $(a, b\in R)$. Then, since $x\notin R$, we have $(a):_R b\subsetneq R$.
Choose $P\in \Ass_R R/[(a):_R b]$. Since multiplication by $b$ induces an injection $R/[(a):_R b]\to R/(a)$, we have $P\in \Ass_R R/(a)$.
Thus $\depth R_P/aR_P=0$, and hence $\depth R_P=1$.
The $(S_2)$ condition gives $P\in \H_1(R)$.
Since $x\in A\subseteq R_P$, we can write $x=\frac{t}{s}=\frac{b}{a}$ $(s\in R\setminus P,~ t\in R)$.
It follows from $sb=ta \in (a)$ that $s\in (a):_R b\subseteq P$, which gives a contradiction.
Therefore $R=A$.
\end{proof}

\begin{prop}\label{2.2}
The following assertions hold.
\benu[$(1)$]
\item
For $X\in \bsF(R)$, we have $\displaystyle X^{v_R}=\bigcap_{P\in \H_1(R)} (X_P)^{v_{R_P}}$.
\item
For $I\in \F(R)$, if $I^{v_R}\neq R$, then $\displaystyle I^{v_R}=\bigcap_{P\in V(I)\cap \H_1(R)}[ (IR_P)^{v_{R_P}}\cap R]$.

\enenu
\end{prop}

\begin{proof}
$(1)$ For every $P\in\H_1(R)$, we have $X^{v_R}\subseteq (X^{v_R})_P=(X_P)^{v_{R_P}}$, and hence $X^{v_R}\subseteq\bigcap_{P\in \H_1(R)} (X_P)^{v_{R_P}}$.
Let $x\in\bigcap_{P\in \H_1(R)} (X_P)^{v_{R_P}}$. For every $P\in \H_1(R)$, we have $x\cdot (R:X)\subseteq x\cdot (R_P: X_P)\subseteq R_P$.
By Lemma~\ref{2.1}, it follows that $x\cdot (R:X)\subseteq \bigcap_{P\in \H_1(R)} R_P=R$.
Thus $x\in X^{v_R}$, which proves $X^{v_R}=\bigcap_{P\in \H_1(R)} (X_P)^{v_{R_P}}$.

$(2)$ Since $I^{v_R}\subseteq R$ and $IR_P=R_P$ whenever $P\notin V(I)$, assertion (1) gives
\[
I^{v_R}
=R\cap\bigcap_{P\in V(I)\cap\H_1(R)}(IR_P)^{v_{R_P}}
=\bigcap_{P\in V(I)\cap\H_1(R)}\bigl((IR_P)^{v_{R_P}}\cap R\bigr).
\]
\end{proof}

\begin{cor}\label{2.3}
For $X, Y\in\bsF^v(R)$, the following conditions are equivalent.
\benu[$(1)$]
\item
$X=Y$.
\item
For every $P\in \H_1(R)$, $X_P=Y_P$.
\enenu
\end{cor}

\begin{proof}
Suppose that $X_P=Y_P$ for every $P\in \H_1(R)$. Then Proposition~\ref{2.2} gives
$$X=X^{v_R}=\bigcap_{P\in \H_1(R)} (X_P)^{v_{R_P}}=\bigcap_{P\in \H_1(R)} (Y_P)^{v_{R_P}}=Y^{v_R}=Y.$$
\end{proof}

\begin{prop}[{cf.\ \cite[Theorem 2.1, Propositions 2.3, 3.3]{HLV}}]\label{2.4}
The following assertions hold.
\benu[$(1)$]
\item
$\H_1(R)\subseteq \F^v(R)$.
\item
For $I\in \F(R)$, we have $\displaystyle \Ass_R R/I^{v_R} = \Min_R R/I^{v_R}=V(I)\cap \H_1(R)$.

\enenu

\end{prop}

\begin{proof}
$(1)$ First, if $(R, \m)$ is a Noetherian local domain of dimension one, then $\m$ is divisorial.
Indeed, for $0\ne a\in\mathfrak m$, the ring $R/aR$ is a nonzero Artinian local ring, so there exists $b\in(aR:\mathfrak m)\setminus aR$.
Since $\frac{b}{a}\in(R:\mathfrak m)\setminus R$, we have $R:\mathfrak m\ne R$.
If $\fkm^{v_R}\ne \fkm$, then $\fkm^{v_R}=R$, and consequently
$$R=R:R=R:\fkm^{v_R}=R:\fkm,$$
which gives a contradiction. Hence $\fkm$ is divisorial.

It follows that, for $P\in\H_1(R)$, we have $(P^{v_R})R_P=(PR_P)^{v_{R_P}}=PR_P$.
Since $P^{v_R}\subseteq R$, this yields $P^{v_R}\subseteq PR_P\cap R=P$, so $P$ is divisorial.

$(2)$ If $I^{v_R}=R$, both sides are empty, so assume that $I^{v_R}\ne R$.
For $P\in V(I)\cap \H_1(R)$, the inclusion $I^{v_R}\subseteq P^{v_R}=P$ gives $P\in V(I^{v_R})$. Therefore $V(I)\cap \H_1(R)\subseteq \Min_R R/I^{v_R}\subseteq \Ass_R R/I^{v_R}$.
The equality $I^{v_R}=\bigcap_{P\in V(I)\cap \H_1(R)}[ (IR_P)^{v_{R_P}}\cap R]$ in Proposition~\ref{2.2} $(2)$ is thus an irredundant primary decomposition of $I^{v_R}$, and we obtain $\Ass_R R/I^{v_R} = \Min_R R/I^{v_R}=V(I)\cap \H_1(R)$.
\end{proof}

\begin{rem}\label{2.5}
For $X, Y\in \bsF(R)$, the following observations will be useful.
\benu[$(1)$]
\item
If $X, Y$ are divisorial, then so is $X\cap Y$.
\item
The fractional ideal $R:X$ is divisorial.
\item
If $X$ is divisorial, then so is $X:Y$.
\enenu
\end{rem}

\begin{lem}[{cf.\ \cite[Proposition 3.2]{HLV}}]\label{2.6}
Let $I\in \F^v(R)$ and $P\in \Ass_R R/I$, and set $I(P)=IR_P\cap R$. Then $I(P)$ is divisorial.

\end{lem}

\begin{proof}
Since $P\in V(I(P))\cap \H_1(R)\subseteq \Min_R R/I(P)\subseteq \Ass_R R/I(P)=\{P\}$, Proposition~\ref{2.2} $(2)$ gives
$$[I(P)]^{v_R}=(I(P)R_P)^{v_{R_P}}\cap R=(IR_P)^{v_{R_P}}\cap R=(I^{v_R})R_P\cap R=IR_P\cap R=I(P).$$
Thus $I(P)$ is divisorial.
\end{proof}

\begin{prop}[{cf.\ \cite[Proposition 2.7]{HLV}}]\label{2.7}
For $R\neq I\in\F^v(R)$, we have
\[
I=\bigcap_{P\in\Ass_R R/I}I(P)
=\left(\prod_{P\in\Ass_R R/I}I(P)\right)^{v_R}.
\]
\end{prop}

\begin{proof}
The first equality follows from Propositions~\ref{2.2} and \ref{2.4}.
For the second, we have $\left( \prod_{P\in \Ass_R R/I} I(P) \right)^{v_R}R_Q=IR_Q$ for every $Q\in\H_1(R)$, so the assertion follows from Corollary~\ref{2.3}.
\end{proof}

\begin{cor}\label{2.8}
Let $R\neq I\in \F^v(R)$. Then, the following assertions hold.
\benu[$(1)$]
\item
We have $\sqrt{I}\in \F^v(R)$ and $\displaystyle \sqrt{I}=\left( \prod_{P\in V(I)\cap \H_1(R)} P \right)^{v_R}$.
\item
If $V(I)\cap\H_1(R)=\{P_1,\ldots,P_n\}\subseteq\H_1^{\rm reg}(R)$, then $\displaystyle I=\left( P_1^{m_1}P_2^{m_2}\cdots P_n^{m_n} \right)^{v_R}$ for some positive integers $m_i$.
\enenu

\end{cor}

\begin{proof}
$(1)$ This follows from Remark~\ref{2.5} $(1)$ and Proposition~\ref{2.7}.

$(2)$ For each $i$, the ring $R_{P_i}$ is a DVR, so we can write $IR_{P_i}=P_i^{m_i}R_{P_i}$ $(m_i>0)$.
Hence $\left( P_1^{m_1}P_2^{m_2}\cdots P_n^{m_n} \right)^{v_R}R_{P}=IR_P$ for every $P\in \H_1(R)$, and Corollary~\ref{2.3} gives the desired equality.
\end{proof}

\section{Strict closedness of finite birational extensions}\label{sec:overrings}

In this section, we study the finite birational extensions
$A\in\B^v(R)$, which play an important role in the description
of the $v$-class semigroup.
We first show that every such $A$ is a Noetherian domain
satisfying $(S_2)$.
We then prove that $A$ is strictly closed whenever $R$ is.

\begin{lem}[{cf.\ \cite{SP1}}]\label{2.9}
Every $X\in \bsF^v(R)$ satisfies $(S_2)$ as an $R$-module.
\end{lem}

\begin{proof}
By clearing denominators, we may assume that $X=I\subseteq R$.
The case $I=R$ is clear, so suppose that $I\subsetneq R$.
Let $P\in\Spec R$. If $\height_R P\le1$, the required depth inequality follows from torsionfreeness.
Suppose that $\height_R P\ge2$. If $(R/I)_P=0$, then $I_P=R_P$ and there is nothing to prove.
Otherwise, Proposition~\ref{2.4} gives $P\notin\Ass_R(R/I)$, so $\depth_{R_P}(R/I)_P\ge1$.
Combining $\depth R_P\ge2$ with the short exact sequence $0\to I_P\to R_P\to(R/I)_P\to0$, we obtain $\depth_{R_P}I_P\ge2$.
\end{proof}

\begin{prop}\label{2.10}
Every $A\in \B^v(R)$ is a Noetherian domain satisfying $(S_2)$ as a ring.

\end{prop}

\begin{proof}
Finite generation implies that $A$ is Noetherian.
Let $Q\in\Spec A$ and $P=Q\cap R$.
If $\height_A Q\le1$, the required depth inequality follows because $A$ is a domain.
Suppose that $\height_A Q\ge2$. Since $R\subseteq A$ is integral, we have $\height_R P\ge\height_A Q\ge2$.
By Lemma~\ref{2.9}, we have $\depth_{R_P}A_P\ge2$, so there exists an $A_P$-regular sequence $x,y$ of length two in $PR_P$.
Localization to $A_Q$ preserves injectivity, and $x,y$ belong to $QA_Q$.
Thus, they form a regular sequence on $A_Q$, so that $\depth A_Q\ge2$.
Therefore, $A$ satisfies $(S_2)$ as a ring.
\end{proof}

\begin{rem}\label{overring-reflexive}
If $A\in\B^v(R)$, then $\bsF^v(A)\subseteq\bsF^v(R)$.
Indeed, for $L\in\bsF^v(A)$, we have
\[
L=A:(A:L)=R:\bigl((R:A)(A:L)\bigr),
\]
so $L$ is $R$-divisorial.
\end{rem}

\begin{lem}\label{3.13}
For $A\in\B^v(R)$ and $Q\in\H_1(A)$, we have $P=Q\cap R\in \H_1(R)$.
\end{lem}

\begin{proof}
Since $A$ is integral over $R$, we have $P\neq (0)$.
If $\height_R P\ge2$, the argument in the proof of Proposition~\ref{2.10} provides a regular sequence of length two on $A_P$ that remains regular on $A_Q$.
This contradicts $\dim A_Q=1$. Hence $\height_R P=1$.
\end{proof}

We then obtain the following.

\begin{thm}\label{3.12}
If $R=R^*$, then $A=A^*$ for every $A\in\B^v(R)$.

\end{thm}

\begin{proof}
Let $Q\in \H_1(A)$ and set $P=Q\cap R$.
By Lemma~\ref{3.13} and Theorem~\ref{3.9}, the ring $R_P$ is Arf and $A_P\in \B^v(R_P)$.
Hence $A_P$ is Arf by Corollary~\ref{3.6}.
The Arf property is preserved under localization of semi-local rings (\cite[Corollary 2.5]{L}, \cite[Remark 2.9]{IK}), so $A_Q$ is also Arf.
Therefore, we obtain $A=A^*$ by Proposition~\ref{2.10} and Theorem~\ref{3.9}, as desired.
\end{proof}

\section{Clifford \texorpdfstring{$v$}{v}-regularity of strictly closed domains}\label{sec:semigroups}
In this section, we study Clifford $v$-regularity and Boole $v$-regularity of strictly closed domains.

The set $\bsF^v(R)$ of nonzero divisorial fractional ideals is a commutative monoid under the $v$-product $X*Y=(XY)^{v_R}$, with identity $R$.
For $X\in\bsF(R)$, we say that $X$ is \emph{invertible} over $R$ if $X(R:X)=R$.
For $X\in\bsF^v(R)$, we say that $X$ is \emph{$v$-invertible} over $R$ if $(X(R:X))^{v_R}=R$.
Set
\begin{align*}
\Inv(R)&=\{X\in\bsF(R)\mid X(R:X)=R\},\\
\Inv^v(R)&=\{X\in\bsF^v(R)\mid (X(R:X))^{v_R}=R\},\\
\Prin(R)&=\{aR\mid a\in K^\times\}.
\end{align*}
Invertible fractional ideals are divisorial, and $\Prin(R)\subseteq\Inv(R)\subseteq\Inv^v(R)$ are inclusions of abelian groups under the $v$-product on $\bsF^v(R)$.
We set $\Cl^v(R)=\Inv^v(R)/\Prin(R)$ and call it the \emph{$v$-class group} of $R$.

Let $\calS^v(R)$ be the set of isomorphism classes of the elements of $\bsF^v(R)$.
Since every isomorphism of fractional ideals is multiplication by an element of $K^\times$, the set $\calS^v(R)$ is a commutative monoid under the product $[X][Y]=[(XY)^{v_R}]$.
This is the \emph{$v$-class semigroup} of $R$.
We identify $\Cl^v(R)$ naturally with the group of units of $\calS^v(R)$ and use $[X]$ to denote the class of $X$ in both cases.
This terminology agrees with that of the $\star$-class semigroup in the case $\star=v$ in \cite{GP}.
Since $t=v$ for Noetherian rings, these are also the $t$-class semigroup and $t$-class group considered in \cite{KM,KMN}.

A commutative semigroup $S$ is called \emph{Clifford} if, for every $x\in S$, there exists $y\in S$ such that $x=x^2y$.
Set $\operatorname{Id}(S)=\{e\in S\mid e^2=e\}$ and
\[
G_e=(eS)^\times
=\{x\in S\mid ex=x, \ ey=y,\ xy=e \text{ for some $y\in S$}\}.
\]
This is the maximal subgroup with identity $e$.
The semigroup $S$ is Clifford if and only if $S=\bigsqcup_{e\in\operatorname{Id}(S)}G_e$.
Also, $S$ is called \emph{Boolean} if $x^2=x$ for every $x\in S$.
This is equivalent to the requirement that $S$ be Clifford and every $G_e$ be a trivial group.
For this terminology concerning class semigroups, see also \cite[\S1]{KM,GP}.

For a domain $R$, we say that $R$ is \emph{Clifford $v$-regular} (resp.\ \emph{Boole $v$-regular}) if $\calS^v(R)$ is Clifford (resp.\ Boolean).
Since isomorphisms between fractional ideals are given by multiplication by nonzero elements of the field of fractions, we can restate these properties as follows.

\benu[$(1)$]
\item
The domain $R$ is Clifford $v$-regular if and only if, for every $I\in \F^v(R)$, there exists $J\in \bsF^v(R)$ such that $I=(I^2J)^{v_R}$.
\item
The domain $R$ is Boole $v$-regular if and only if, for every $I\in \F^v(R)$, there exists $\alpha\in K^{\times}$ such that $I=(\alpha I^2)^{v_R}$.
\enenu

\Needspace{8\baselineskip}
We then have the following theorem.

\begin{thm}\label{5.5}
The following conditions are equivalent.
\benu[$(1)$]
\item
$R=R^*$.
\item
$R$ is Clifford $v$-regular.
\item
For every $I\in \F^v(R)$, setting $A=I:I$ gives $I\in \Inv^v(A)$.
\enenu
When these conditions hold, setting $\calS_A^v(R)=\{[I]\in\calS^v(R)\mid I:I=A\}$ for each $A\in\B^v(R)$, we have
\[
\calS^v(R)=\bigsqcup_{A\in\B^v(R)}\calS_A^v(R) \quad \text{and}
\quad \calS_A^v(R)\cong\Cl^v(A),
\]
where the latter is a natural isomorphism of abelian groups, and the identity of the corresponding constituent group is $[A]$.
\end{thm}

\begin{rem}
In \cite[\S1]{GP}, for $I=I^{v_R}$ and $A=I:I$, the ideal $I$ is called {\it $v$-stable} when
\[
\bigl(I(A:I)\bigr)^{v_R}=A.
\]
Here we use the restriction of the $v_R$-operation to $A$, which is not identified a priori with the $v_A$-operation of $A$ itself.
In our setting, Proposition~\ref{5.4} and Theorem~\ref{5.5} show that imposing the displayed condition on every $I\in\F^v(R)$ is equivalent to condition (3) of Theorem~\ref{5.5}.
The relationship between stability and Clifford regularity with respect to general star operations is discussed in \cite[Lemma 1.2, Proposition 1.5]{GP}.
Local criteria for Clifford regularity and the equivalence between invertibility and being principal in the one-dimensional local setting are given in \cite[Propositions 2.5(1), 3.8]{GP2}.
Moreover, in our setting of $(S_2)$ domains, maximal $t$-ideals have height one unless $R$ is a field, so the equivalence of $(2)$ and $(3)$ in Theorem~\ref{5.5} also corresponds to \cite[Theorem 2.2]{KMN}.
For earlier work describing constituent groups as class groups of overrings, see \cite[Theorem 2.1(1)]{KMC}.
That result concerns $t$-class semigroups of Krull-type domains and does not apply directly to the present setting, which includes non-integrally closed rings.
We give a direct proof of Theorem~\ref{5.5} in our setting.
\end{rem}

\begin{lem}[cf.\ {\cite[Lemma 1.2]{GP}}]\label{5.2}
Let $I\in \F^v(R)$ and set $A=I:I$. Then, the following conditions are equivalent.
\benu[$(1)$]
\item
There exists $J\in \bsF^v(R)$ such that $I=(I^2J)^{v_R}$.
\item
$I=(I^2(I:I^2))^{v_R}$.
\item
$I=(I^2(A:I))^{v_R}$.

\enenu
\end{lem}

\begin{proof}
Since $I:I^2=(I:I):I=A:I$, we have $(2)\Leftrightarrow (3)$. The implication $(2)\Rightarrow (1)$ is clear.

$(1)\Rightarrow (2):$ Suppose that $I=(I^2J)^{v_R}$ for some $J\in \bsF^v(R)$.
Since $I^2J\subseteq (I^2J)^{v_R}=I$, we have $J\subseteq I:I^2$.
Therefore, $I=(I^2J)^{v_R}\subseteq (I^2(I:I^2))^{v_R}\subseteq I^{v_R}=I$, which yields $I=(I^2(I:I^2))^{v_R}$.
\end{proof}

\begin{lem}[{cf.\ \cite[Proposition 3.8]{GP2}}]\label{5.3}
Let $(R,\fkm)$ be a one-dimensional Noetherian local domain, let $I\in\F^v(R)$, and set $A=I:I$.
Then, the following conditions are equivalent.
\benu[$(1)$]
\item
$(I(A:I))^{v_A}=A$.
\item
$I(A:I)=A$.
\item
There exists $a\in I$ such that $I^2=aI$.
\enenu

\end{lem}

\begin{proof}
The implication $(2)\Rightarrow (1)$ is clear.

$(1)\Rightarrow (2): $ Suppose that $(I(A:I))^{v_A}=A$ and $I(A:I)\subsetneq A$.
Choose $M\in \Max A$ such that $I(A: I)\subseteq M$.
Since $\height_A M=1$, we have $M^{v_A}=M$, and hence $A=(I(A:I))^{v_A}\subseteq M^{v_A}=M$, which is a contradiction.
Thus, $I(A:I)=A$.

$(3)\Rightarrow (2):$ If $I^2=aI$, then $A=\frac{I}{a}$.
Hence, $I=aA$, so $I$ is a principal ideal of $A$, which yields $I(A:I)=A$.

$(2)\Rightarrow (3):$ Since $I$ is invertible over $A$, the module $IA_M$ is a free $A_M$-module of rank one for every $M\in \Max A$.
Since $R$ is local, $A$ is semi-local, so $I$ is a free $A$-module of rank one.
Thus, we can write $I=aA$ for some $a\in I$, which yields $I^2=aI$, as desired.
\end{proof}

\begin{lem}\label{5.a}
For $X\in \bsF^v(R)$, the following conditions are equivalent.

\benu[$(1)$]
\item
$(X(R:X))^{v_R}=R$.
\item
$X:X=R$.
\item
For every $P\in \H_1(R)$, we have $X_P\cong R_P$.

\enenu

\end{lem}

\begin{proof}
$(2)\Leftrightarrow (1):$ Set $Y=X(R:X)$. Since $X$ is divisorial, we have
$R:Y=(R:(R:X)):X=X:X$.
If $X:X=R$, then $Y^{v_R}=R:(R:Y)=R$.
Conversely, if $Y^{v_R}=R$, then $X:X=R:Y=R:Y^{v_R}=R$.

$(1)\Rightarrow (3):$ Let $P\in\H_1(R)$. We have $(X_P(R_P:X_P))^{v_{R_P}}=R_P$.
Since $PR_P$ is divisorial, the proper inclusion $X_P(R_P:X_P)\subsetneq R_P$ would imply $(X_P(R_P:X_P))^{v_{R_P}}\subseteq PR_P$, which is a contradiction.
Hence, $X_P$ is invertible. By the same argument as in the proof of $(2)\Rightarrow (3)$ in Lemma~\ref{5.3}, we obtain that $X_P$ is principal.

$(3)\Rightarrow (2):$ For every $P\in \H_1(R)$, we have $(X:X)_P=X_P:X_P=R_P:R_P=R_P$. By Corollary~\ref{2.3}, we obtain $X:X=R$.
\end{proof}

\begin{prop}[cf.\ {\cite[Proposition 1.5]{GP}}]\label{5.4}
Let $I\in \F^v(R)$ and set $A=I:I$. Then, the following conditions are equivalent.
\benu[$(1)$]
\item
$I=(I^2(A:I))^{v_R}$.
\item
$(I(A:I))^{v_R}=A$.
\item
$(I(A:I))^{v_A}=A$.
\enenu

\end{prop}

\begin{proof}
$(3)\Rightarrow (2):$ Suppose that $(I(A:I))^{v_A}=A$. For every $P\in \H_1(R)$, we have $(IR_P(A_P:IR_P))^{v_{A_P}}=((I(A:I))^{v_A})_P=A_P$, so Lemma~\ref{5.3} gives $IR_P(A_P:IR_P)=A_P$.
Also, $((I(A:I))^{v_R})_P=(IR_P(A_P:IR_P))^{v_{R_P}}=(A_P)^{v_{R_P}}=A_P$, and hence
$$((I(A:I))^{v_A})_P=A_P=((I(A:I))^{v_R})_P.$$
Since $(I(A:I))^{v_R}, (I(A:I))^{v_A}\in \bsF^v(R)$, Corollary~\ref{2.3} yields $(I(A:I))^{v_R}=(I(A:I))^{v_A}=A$.

$(2)\Rightarrow (1):$ Suppose that $(I(A:I))^{v_R}=A$. Computing the product in $\bsF^v(R)$, we obtain $(I^2(A:I))^{v_R}=I*(I(A:I))^{v_R}=I*A=(IA)^{v_R}=I^{v_R}=I$.

$(1)\Rightarrow (3):$ Set $X=I^2(A:I)$. Then $I:X=I:X^{v_R}$.
Indeed, $uX\subseteq I$ implies $uX^{v_R}\subseteq I^{v_R}=I$. Since $I=X^{v_R}$, we obtain
\[
\begin{aligned}
A&= I:X^{v_R}\\
& = I:X\\
& =I:I^2(A:I)\\
& =(I:I):I(A:I)\\
& =A:I(A:I),
\end{aligned}
\]
so that $A=A:A=A:(A:I(A:I))=(I(A:I))^{v_A}$, as desired.
\end{proof}

\begin{prop}\label{5.10}
Suppose that $R=R^*$, and let $A\in \B^v(R)$. Then $X^{v_A}=X^{v_R}$ for every $X\in \bsF(A)$.

\end{prop}

\begin{proof}
Since $X^{v_R}, X^{v_A}\in \bsF^v(R)$, we may, by Corollary~\ref{2.3} and Theorem~\ref{3.9}, localize at $P\in \H_1(R)$ and assume that $(R, \m)$ is an Arf local domain of dimension one.
By Theorem~\ref{3.3}, we have $R:A\cong R:(R:A)=A$, so we can write $A=\alpha (R:A)$ $(\alpha \in K^{\times})$.
Hence
\[
\begin{aligned}
X^{v_A} & =A:(A:X)\\
& =\alpha(R:A):(A:X)\\
& =\alpha[(R:A):(A:X)]\\
& =\alpha[R:A(A:X)]\\
& =\alpha[R:(A:X)]\\
& =\alpha[R:(\alpha(R:A):X)]\\
& =\alpha[R:\alpha[(R:A):X]]\\
& =R:((R:A):X)\\
& =R:(R:AX)\\
& =R:(R:X)\\
& =X^{v_R}.
\end{aligned}
\]
\end{proof}

\begin{proof}[Proof of Theorem $\ref{5.5}$]

$(3)\Rightarrow (2):$ This follows from Lemma~\ref{5.2} and Proposition~\ref{5.4}.

$(2)\Rightarrow (1):$ Let $P\in \H_1(R)$. Choose $\fka\in \F^v(R_P)$ and write $\fka=(f_1, f_2, \cdots, f_n)$ $(f_i\in R_P)$.
Set $X=\sum_{i=1}^n Rf_i$. Then $X\in \bsF(R)$ and $[X^{v_R}]_P=\fka^{v_{R_P}}=\fka$.
Choose $0\neq a\in R$ such that $aX^{v_R}\subseteq R$, and set $I=aX^{v_R}$.
For $A=I:I$, assumption $(2)$, Lemma~\ref{5.2}, and Proposition~\ref{5.4} give $(I(A:I))^{v_A}=A$.
Noting that
$$A=aX^{v_R}:aX^{v_R}=X^{v_R}:X^{v_R},$$
$$A=(I(A:I))^{v_A}=(aX^{v_R}(A:aX^{v_R}))^{v_A}=(X^{v_R}(A:X^{v_R}))^{v_A},$$
we obtain $A_P=[X^{v_R}]_P:[X^{v_R}]_P=\fka:\fka$ and $A_P=([X^{v_R}]_P(A_P:[X^{v_R}]_P))^{v_{A_P}}=(\fka(A_P:\fka))^{v_{A_P}}$.
By Lemma~\ref{5.3}, there exists $\alpha\in\fka$ such that $\fka^2=\alpha \fka$.
Thus, $R_P$ is Arf by Theorem~\ref{3.4}, and Theorem~\ref{3.9} gives $R=R^*$.

$(1)\Rightarrow (3):$ Let $I\in \F^v(R)$ and set $A=I:I$.
For every $P\in \H_1(R)$, Theorem~\ref{3.9} shows that $R_P$ is Arf.
Thus, Theorem~\ref{3.4} gives $a\in IR_P$ such that $(IR_P)^2=aIR_P$.
By Lemma~\ref{5.3}, we have
$$[(I(A:I))^{v_A}]_P=(IR_P(A_P:IR_P))^{v_{A_P}}=A_P.$$
Since $(I(A:I))^{v_A}, A\in \bsF^v(R)$, Corollary~\ref{2.3} gives $(I(A:I))^{v_A}=A$.
The equality $A=I:I$ shows that $I$ is an ideal of $A$, and Proposition~\ref{5.10} gives $I=I^{v_R}=I^{v_A}$.
Hence, $I\in \Inv^v(A)$.

Finally, we prove the decomposition into constituent groups.
For every $X\in\bsF^v(R)$, the ring $X:X$ is finite over $R$ since
$X:X\cong\End_R(X)$ as $R$-modules, and it is $R$-divisorial by
Remark~\ref{2.5} $(3)$.
Thus, $X:X\in\B^v(R)$.
Moreover, $(aX):(aX)=X:X$ for every $a\in K^\times$.
The above equivalence also applies to fractional ideals and gives
\[
\Inv^v(A)=\{X\in\bsF^v(R)\mid X:X=A\}
\]
for any $A\in \B^v(R)$.
The inclusion from right to left follows from (3), and the reverse inclusion follows from Remark~\ref{overring-reflexive} and Lemma~\ref{5.a} applied over $A$.
Thus, $\bsF^v(R)$ is the disjoint union of these sets.
For ideals in the same $A$-component, isomorphisms as $R$-modules and as $A$-modules are both given by multiplication by elements of $K^\times$.
Moreover, the $v_R$-product and the $v_A$-product agree by Proposition~\ref{5.10}.
Therefore, by passing to isomorphism classes, we identify each component with $\Cl^v(A)$. This completes the proof.
\end{proof}

\begin{ex}[cf.\ {\cite[I, Example 3.10.1]{W}}]\label{ex:affine-cusp}
Let $k$ be a field, $V=k[t]$ be a polynomial ring over $k$, and $R=k[t^2,t^3]=k+t^2V\subseteq V$. Then $R=R^*$ and
\[
\B^v(R)=\{R,V\},\qquad
 \Cl^v(R)\cong (k,+),\quad
 \calS^v(R)\cong(k,+)\sqcup\{\omega\}.
\]
For $a\in k$, set
\[
 L_a=k(1+at)+t^2V.
\]
Under the isomorphism, $[L_a]$ corresponds to $a$ in $(k,+)$, and $\omega=[V]$.
The multiplication is given by
\[
 [L_a][L_b]=[L_{a+b}],\qquad [L_a]\omega=\omega,\qquad\omega^2=\omega.
\]
\end{ex}

\begin{proof}
The ring $R$ is a one-dimensional domain.
Set $M=(t^2,t^3)R$. Then $R_M$ is a Cohen-Macaulay local ring of multiplicity two, and hence is Arf (\cite[Example, p.~664]{L}).
The quotient $V/R$ is a one-dimensional $k$-vector space, and $R:V=t^2V$ and $R:(t^2V)=V$, so $\B^v(R)=\{R,V\}$.
Set $C=t^2V$. Since $\H_1(R)\setminus \H^{{\rm reg}}_1(R)=V(C)=\{M\}$, Theorem~\ref{3.9} gives $R=R^*$.

For $a\in k$, we have $L_aV=(1+at,t^2)V=V$.
Thus, for any $a,b\in k$, we have $C=VC=L_aVC=L_aC\subseteq L_aL_b$.
Since $(1+at)(1+bt)\equiv1+(a+b)t\pmod C$ in $V/C$, it follows that $L_aL_b=L_{a+b}$.
Hence, $L_aL_{-a}=L_0=R$, so each $L_a$ is invertible.
If $L_a=cL_b$ $(c\in k(t)^\times)$, then $V=cV$ gives $c\in k^\times$.
By comparing the constant terms and the coefficients of $t$ of the generators in $V/C$, we obtain $a=b$.

Conversely, let $L\in\Inv^v(R)$.
Since $R$ is one-dimensional, Lemma~\ref{5.a} shows that $L$ is principal after localization at every maximal ideal, and hence is invertible.
Since $V$ is a PID, after replacing $L$ by an isomorphic fractional ideal, we may assume that $LV=V$.
Then $L\subseteq V$, and $LC=t^2LV=C$ together with $C\subseteq R$ gives $C\subseteq L$.
If $L=C$, then $LV=C\ne V$.
If $L=V$, invertibility of $L$ gives $R=L:L=V:V=V$, which is again a contradiction.
Thus
\[
 C\subsetneq L\subsetneq V.
\]
Since $V/C$ is a two-dimensional $k$-vector space with basis $1,t$, the quotient $L/C$ is one-dimensional.
A nonzero generator of this quotient has a nonzero constant term, since otherwise $L\subseteq tV$, contradicting $LV=V$.
After division by an element of $k^\times$, we can therefore write $L=k(1+at)+C=L_a$ for some $a\in k$.
It follows that $a\mapsto[L_a]$ is a group isomorphism from $(k,+)$ to $\Cl^v(R)$.
The description of the class semigroup and its multiplication follows from $\Cl^v(V)=0$, $L_aV=V$, and Theorem~\ref{5.5}.
\end{proof}

\begin{ex}\label{ex:pinch}
Let $k$ be a field and $V=k[[s,t]]$ be the formal power series ring over $k$. For $n\ge2$, set $R_n=k[[s^n]]+tV\subseteq V$.
Then $ R_n=R_n^*$ and
\[
\B^v(R_n)=\{R_n,V\},\quad
 \Cl^v(R_n)\cong\mathbb Z/n\mathbb Z,\quad
 \calS^v(R_n)\cong(\mathbb Z/n\mathbb Z)\sqcup\{\omega\}.
\]
Here $\omega=[V]$, and multiplication within the group component is addition.
When $n=2$, we have $R_2\cong k[[u,v,w]]/(w^2-uv^2)$.
\end{ex}

\begin{proof}
Throughout this proof, we write $R=R_n$.
The ring $R$ is a finite free module over $k[[s^n,t]]$ with basis $1,st,\ldots,s^{n-1}t$.
Thus, it is a two-dimensional Cohen-Macaulay local domain whose integral closure $V$ is finite as an $R$-module and whose conductor is $\fkp=tV\in\H_1(R)$.
In particular, $V\in \B^v(R)$.
Set $T=R_\fkp$ and $W=V_\fkp=V_{(t)}$. Then $\fkp T=tW$, $(\fkp T)^2=t(\fkp T)$, and $\fkp T:\fkp T=W$.
Consequently, $T$ has minimal multiplicity, and its blow-up at $\fkp T$ is the DVR $W$.
It follows from \cite[Theorem 2.2]{L} that $T$ is Arf.
Since $\H_1(R)\setminus\H_1^{\rm reg}(R)=V(\fkp)\cap\H_1(R)=\{\fkp\}$, Theorem~\ref{3.9} gives $R=R^*$.
On the other hand, if $R\subsetneq A\in\B^v(R)$, then $\fkp=R:V\subseteq R:A\subsetneq R$.
The ideal $R:A$ is divisorial, so Proposition~\ref{2.4} (2) provides a height-one prime ideal $P$ containing $R:A$.
From $\fkp\subseteq R:A\subseteq P$ we obtain $P=\fkp$, and hence $R:A=\fkp$.
Therefore, $A=R:(R:A)=R:\fkp=V$, proving $\B^v(R)=\{R,V\}$.

We compute the class group using Proposition~\ref{5.8} below.

Since $\H_1(R)\setminus\H_1^{\rm reg}(R)=V(\fkp)\cap\H_1(R)=\{\fkp\}$, every element of $\H_1^{\rm reg}(R)$ can be written as $P_f=fV\cap R$ for an irreducible element $f$ of $V$ satisfying $f(s,0)\ne0$, and $R_{P_f}=V_{fV}$.

Set $D=k[[s^n]]$ and take
\[
 Q=sV\cap R=s^nD+stV.
\]
Then $Q=P_s\in \H_1^{\rm reg}(R)$.
Since $Q$ is divisorial and principal at every height-one prime ideal, Lemma~\ref{5.a} gives $Q\in\Inv^v(R)$.
For every $m\ge1$, $Q$ is the only height-one prime ideal containing $Q^m$.
As $R_Q=V_{sV}$ is a DVR, Proposition~\ref{2.2} (2) yields
\[
 (Q^m)^{v_R}=Q^mR_Q\cap R=s^mV_{sV}\cap R=s^mV\cap R.
\]
In particular, $D\cap s^mV=s^nD$ for $1\le m\le n$, so
\[
 (Q^m)^{v_R}=s^nD+s^mtV\qquad(1\le m\le n).
\]
Let $P_f\in\H_1^{\rm reg}(R)$. Since $f(s,0)\ne0$, we can write $f(s,0)=s^mu(s)$ $(m\ge1$, $u(s)\in k[[s]]^\times)$.
Replacing $f$ by $u(s)^{-1}f$ does not change $P_f$, so we may assume that $f-s^m\in tV$.
For every $a\in V$, we have $(f-s^m)a\in tV\subseteq R$, and hence $fa\in R$ if and only if $s^ma\in R$.
Therefore
\[
 \frac{P_f}{f}=\frac{fV\cap R}{f}=\frac{s^mV\cap R}{s^m}=\frac{(Q^m)^{v_R}}{s^m},
\]
which gives $[P_f]=[Q]^m$.
Proposition~\ref{5.8} now yields $\Cl^v(R)=\langle[Q]\rangle$.

Moreover, $(Q^n)^{v_R}=s^nD+s^ntV=s^nR$, so $[Q]^n=[R]$.
For $1\le m<n$, however,
\[
 \bigl((Q^m)^{v_R}\bigr)V=(s^n,s^mt)V=s^m(s^{n-m},t)V
\]
is not a principal ideal of $V$.
Hence, $(Q^m)^{v_R}$ is not principal over $R$, and $[Q]^m\ne[R]$.
Thus, $[Q]$ has order $n$, giving $\mathbb{Z}/n\mathbb{Z} \cong \Cl^v(R)$.
The description of the class semigroup follows from $\Cl^v(V)=0$ and Theorem~\ref{5.5}.
For any $X\in\bsF^v(R)$, Proposition~\ref{5.10} gives $(XV)^{v_R}=(XV)^{v_V}$.
Since $V$ is a UFD, this is a principal fractional ideal of $V$. Hence $[X][V]=[V]$.
The isomorphism for $n=2$ is given by $u\mapsto s^2$, $v\mapsto t$, and $w\mapsto st$.
\end{proof}

\begin{lem}[{cf.\ \cite[Proposition 5.3]{GKR}}]\label{5.11a}
The following assertions hold.
\benu[$(1)$]
\item For $A\in\B^v(R)$, the map $\phi_A:\Inv^v(R)\to\Inv^v(A)$ defined by $X\mapsto(XA)^{v_A}$ is a surjective group homomorphism.
It therefore induces a surjective group homomorphism $\Cl^v(R)\to\Cl^v(A)$.
\item Let $W\subseteq R\setminus\{0\}$ be multiplicatively closed.
The map $\phi_W:\Inv^v(R)\to\Inv^v(W^{-1}R)$ defined by $X\mapsto W^{-1}X$ is a surjective group homomorphism.
It therefore induces a surjective group homomorphism $\Cl^v(R)\to\Cl^v(W^{-1}R)$.

\enenu
\end{lem}

\begin{proof}
First, we prove the following claim.

\begin{claim}\label{clm:patching}
For each $P\in\H_1(R)$, choose $\alpha_P\in K^\times$ so that $\alpha_P=1$ for all but finitely many prime ideals.
Set
\[
J=\bigcap_{P\in\H_1(R)}\alpha_PR_P.
\]
Then $J\in\Inv^v(R)$ and $J_P=\alpha_PR_P$.

\end{claim}

\begin{proof}[Proof of Claim~\ref{clm:patching}.]
Set $F=\{P\in \H_1(R) \mid \alpha_P\ne1\}$ and write $\alpha_P=\frac{r_P}{s_P}$ $(r_P,s_P\in R\setminus\{0\})$.
Let $e=\prod_{P\in F}r_P$ and $d=\prod_{P\in F}s_P$. Lemma~\ref{2.1} gives $eR\subseteq J\subseteq d^{-1}R$, so $J\in \bsF(R)$.

Fix $P$ and choose $0\ne x\in\alpha_PR_P$. Write $x=\frac{u}{v}$ $(u,v\in R)$ and set
\[
F_x=\{Q\in\H_1(R)\mid x\notin\alpha_QR_Q\}.
\]
Then $F_x\subseteq F\cup(V(v)\cap\H_1(R))$.
Indeed, if $Q\in \H_1(R)$ satisfies $\alpha_Q=1$ and $v\notin Q$, then $x\in R_Q=\alpha_QR_Q$.
Thus, $F_x$ is finite, and $P\notin F_x$.

For each $Q\in F_x$, choose $t_Q\in Q\setminus P$.
Since $(\alpha_QR_Q:_{R_Q}x)$ is $QR_Q$-primary, we have $t_Q^{n_Q}x\in\alpha_QR_Q$ for some $n_Q>0$.
Set $t=\prod_{Q\in F_x}t_Q^{n_Q}$. Then $t\notin P$ and $tx\in J$, so $x\in J_P$.
The inclusion $J_P\subseteq\alpha_PR_P$ is clear, and hence $J_P=\alpha_PR_P$.
Moreover,
$$J^{v_R}\subseteq\bigcap_{P\in \H_1(R)}(J_P)^{v_{R_P}}=\bigcap_{P\in \H_1(R)}(\alpha_PR_P)^{v_{R_P}}=J$$
shows that $J$ is divisorial. Lemma~\ref{5.a} now gives $J\in\Inv^v(R)$.
\end{proof}

$(1)$ Let $X\in\Inv^v(R)$, $Q\in\H_1(A)$, and $P=Q\cap R$.
By Lemmas~\ref{3.13} and \ref{5.a}, we can write $X_P=\alpha R_P$, and hence
$$((XA)^{v_A})_Q=((X_PA_P)^{v_{A_P}})_Q=((\alpha A_P)^{v_{A_P}})_Q=(\alpha A_P)_Q=\alpha A_Q.$$
Thus, $(XA)^{v_A}\in\Inv^v(A)$.
By the same calculation at height-one prime ideals, we obtain
\[
\phi_A((XY)^{v_R})=(\phi_A(X)\phi_A(Y))^{v_A},
\]
so $\phi_A$ is a group homomorphism.

Let $L\in\Inv^v(A)$. Choose $b,c\in R\setminus\{0\}$ such that $cA\subseteq L\subseteq b^{-1}A$.
If $P\notin V(bc)$, then $L_P=A_P$.
For the remaining $P\in\H_1(R)\cap V(bc)$, the ring $A_P$ is a one-dimensional semi-local domain, and $L_P$ is principal after localization at each of its maximal ideals.
Thus, we can write $L_P=\alpha_PA_P$.
Choosing $\alpha_P=1$ for all but finitely many prime ideals, let $J$ be the fractional ideal constructed in Claim~\ref{clm:patching}.
Then
\[
((JA)^{v_A})_P=\alpha_PA_P=L_P.
\]
Since both sides are $R$-divisorial, we obtain $(JA)^{v_A}=L$. Hence the map is surjective.
The principal fractional ideal $aR$ maps to $aA$, so the induced map on class groups is also surjective.

$(2)$ Set $T=W^{-1}R$. For $X\in \Inv^v(R)$, we have $(XT(T:XT))^{v_{T}}=((X(R:X))^{v_R})T=T$, so $XT\in\Inv^v(T)$.
Localization commutes with $v$-closure and therefore preserves the
$v$-product. Hence, $\phi_W$ is a group homomorphism.
Let $L\in\Inv^v(T)$. By clearing denominators, we may assume that $L\subseteq T$.
Choose $0\ne\ell\in L$ and write $\ell=\frac{a}{w}$ ($a\in R\setminus\{0\}$, $w\in W$). Then $aT=\ell T\subseteq L\subseteq T$.
Consequently, if $P\in \H_1(R)$, $P\cap W=\emptyset$, and $a\notin P$, then $L_{PT}=R_P$.
For $P\in\H_1(R)$ satisfying $P\cap W=\emptyset$, write $L_{PT}=\alpha_PR_P$, choosing $\alpha_P=1$ when $a\notin P$.
Since $V(a)\cap\H_1(R)$ is finite, only finitely many $P$ satisfy $\alpha_P\ne1$.
For $P\in \H_1(R)$ with $P\cap W\ne \emptyset$, set $\alpha_P=1$.
The construction in Claim~\ref{clm:patching} then gives $J\in\Inv^v(R)$ and $JT=L$.
If denominators were cleared at the outset, we divide the resulting fractional ideal by the same nonzero element.
Since principal fractional ideals map to principal fractional ideals, the induced map on class groups is also surjective. This completes the proof.
\end{proof}

\begin{prop}\label{5.13}
Suppose that $R=R^*$. Then, the following conditions are equivalent.

\benu[$(1)$]
\item
The set $\B^v(R)$ is finite.
\item
 $\bar{R}$ is a finitely generated $R$-module.

\enenu

\end{prop}

\begin{proof}
$(1)\Rightarrow(2):$ For $x\in\bar R$, set $B=R[x]$ and $\mathfrak d=R:B$.
Then $B\in\B(R)$ and $\mathfrak d=\mathfrak d B=\tr_R(B)$, so
\[
 B^{v_R}=R:\mathfrak d=\mathfrak d:\mathfrak d.
\]
The latter is a finite birational extension of $R$, since $\mathfrak d$ is a nonzero finitely generated ideal.
Thus, $B^{v_R}\in\B^v(R)$ and $x\in B^{v_R}\subseteq\bar R$.
It follows that $\bar R=\bigcup_{A\in\B^v(R)}A$.
If $\B^v(R)$ is finite, then $\bar R=\sum_{A\in\B^v(R)}A$ is a finitely generated $R$-module.

$(2)\Rightarrow(1):$ Suppose that $\bar{R}$ is finitely generated as an $R$-module, and set $\fkc=R:\bar{R}$.
For $A\in \B^v(R)$, set $I_A=R:A$. Then $I_A\in \F^v(R)$ is a trace ideal.
By Theorem~\ref{3.3}, the ideal $I_AR_P$ is integrally closed for every $P\in \H_1(R)$, so $I_A$ is integrally closed.
If $A, B\in \B^v(R)$ satisfy $I_A=I_B$, then $A=R:I_A=R:I_B=B$.
To prove that $\B^v(R)$ is finite, it therefore suffices to show that there are only finitely many integrally closed ideals $I\in \F^v(R)$ satisfying $\fkc\subseteq I$.

Choose $P\in \Ass_R R/\fkc$ and set $T=R_P$, $V=(\bar R)_P$, and $\mathfrak d=\mathfrak c_P$.
The ring $V$ is a one-dimensional Noetherian normal semi-local domain.
Let $\mathfrak n_1,\ldots,\mathfrak n_r$ be its maximal ideals, and write $\mathfrak d=\prod_j\mathfrak n_j^{c_j}$.
For an integrally closed ideal $L$ satisfying $\mathfrak d\subseteq L\subseteq T$, we have $L=LV\cap T$ and
\[
LV=\prod_j\mathfrak n_j^{b_j} \qquad (0\le b_j\le c_j).
\]
Thus, there are at most $\prod_j(c_j+1)$ choices for $L$.
Consequently, for every integrally closed ideal $I\in \F^v(R)$ satisfying $\fkc\subseteq I$, there are only finitely many possibilities for $IR_P$.
For $P\in\H_1(R)\setminus\Ass_R(R/\fkc)$, we have $IR_P=R_P$, since $\fkc\subseteq I$ and $\fkc\nsubseteq P$.
Since $\Ass_R R/\fkc$ is finite, Corollary~\ref{2.3} shows that there are only finitely many such ideals $I$.
Therefore, $\B^v(R)$ is finite.
\end{proof}

\begin{thm}\label{5.14}
Suppose that $R=R^*$. Then, the following conditions are equivalent.

\benu[$(1)$]
\item
The monoid $\calS^v(R)$ is finitely generated.
\item
$\bar R$ is a finitely generated $R$-module, and $\Cl^v(R)$ is a finitely generated abelian group.

\enenu

\end{thm}

\begin{proof}
$(1)\Rightarrow (2):$ Let $s_1, s_2, \ldots, s_m$ be generators of $\calS^v(R)$, and let $e_1,\ldots,e_m$ be the corresponding idempotents.
The idempotent corresponding to $s_is_j$ is $e_ie_j$.
Since each $e_i$ is idempotent, the semigroup $\calS^v(R)$ has at most $2^m$ idempotents.
By Theorem~\ref{5.5}, the set $\B^v(R)$ is finite. Hence $\bar R$ is finitely generated as an $R$-module by Proposition~\ref{5.13}.
In a commutative monoid, if a product is a unit, then each factor is a unit.
Hence, $\Cl^v(R)$ is generated by those $s_i$ that are units.

$(2)\Rightarrow (1):$ By Proposition~\ref{5.13}, we can write $\B^v(R)=\{A_1,\ldots,A_n\}$.
Let $g_1,\ldots,g_m$ generate $\Cl^v(R)$ as a group.
By Lemma~\ref{5.11a} $(1)$ and Proposition~\ref{5.10}, every element of the $A_i$-component has the form $g[A_i]$ with $g\in\Cl^v(R)$.
Thus, $\calS^v(R)$ is generated by $[A_1],\ldots,[A_n]$ together with $g_j,g_j^{-1}$. This completes the proof.
\end{proof}

\begin{lem}\label{5.7}
Suppose that $R=R^*$ and that $\bar{R}$ is a finitely generated $R$-module.
Then, for every $I\in \F^v(R)$, there exists an integrally closed ideal $J\in \F^v(R)$ such that $I\cong J$.

\end{lem}

\begin{proof}
Set $\Sigma=\H_1(R)\setminus \H_1^{{\rm reg}}(R)$.
Since $\bar{R}$ is finitely generated, setting $\fkc=R:\bar{R}$ gives $\fkc\in \F^v(R)$ and $\Sigma \subseteq V(\fkc)\cap \H_1(R)=\Ass_R R/\fkc$, so $\Sigma$ is finite.
If $\Sigma=\emptyset$, the localizations of $I$ at height-one prime ideals are ideals of DVRs and hence are integrally closed.
By Proposition~\ref{2.2} $(1)$, the ideal $I$ is integrally closed, so we may take $J=I$.

Suppose that $\Sigma\ne\emptyset$.
Set $W=R\setminus\bigcup_{P\in\Sigma}P$ and $T=W^{-1}R$. Then $T$ is a one-dimensional Arf semi-local domain.
By Theorem~\ref{3.3},
\[
IT\cong \tr_T(IT)=\tr_R(I)T
\]
and the right-hand side is integrally closed.
Choose $\alpha\in K^\times$ such that $\alpha IT=\tr_R(I)T$, and then choose $s\in W$ such that $s\alpha I\subseteq R$.
Set $J=s\alpha I$. Then $J\in\F^v(R)$.
For $P\in\Sigma$, the equality $JR_P=[s\alpha IT]_P=[\tr_R(I)T]_P=\tr_R(I)R_P$ shows that $JR_P$ is integrally closed.
For every $P\in \H_1(R)\setminus \Sigma$, the ring $R_P$ is a DVR, so $JR_P$ is again integrally closed.
By Proposition~\ref{2.2} $(1)$, we have $J=\bigcap_{P\in\H_1(R)}JR_P$.
It follows that $J$ is integrally closed and $I\cong J$.
\end{proof}

\begin{prop}\label{5.8}
Suppose that $R=R^*$ and that $\bar{R}$ is a finitely generated $R$-module. Then
$$\Cl^v(R)=\langle [P] \mid P\in \H_1^{\rm reg}(R) \rangle.$$

\end{prop}

\begin{proof}
Every $P\in\H_1^{\rm reg}(R)$ is divisorial.
Its localization at $P$ is a principal ideal of a DVR, and its localization at every other height-one prime ideal is the unit ideal.
Lemma~\ref{5.a} therefore gives $P\in\Inv^v(R)$.

Let $X\in\Inv^v(R)$. By clearing denominators, we may assume that $X\subseteq R$.
By Lemma~\ref{5.7}, choose an integrally closed ideal $J\in\F^v(R)$ such that $X\cong J$.
If $J=R$, then $[X]=[R]$ and there is nothing to prove. Therefore, we may assume that $J\neq R$.
For $P\in\Ass_R(R/J)$, we can write $JR_P=aR_P$ $(0\ne a\in R)$.
We claim that $R_P$ is a DVR.
Indeed, since $a\in PR_P\setminus\{0\}$ and $\dim R_P=1$, for every $z\in\bar{R_P}$ there is a least $n\ge0$ such that $a^nz\in R_P$.
If $n>0$, then $a^nz\in a\bar{R_P}\cap R_P=\overline{aR_P}=aR_P$, contradicting minimality.
Thus, $z\in R_P$, so $R_P$ is integrally closed and hence is a DVR.
Consequently, $\Ass_R R/J\subseteq \H^{{\rm reg}}_1(R)$, and Corollary~\ref{2.8} $(2)$ gives $J=(P_1^{m_1}P_2^{m_2}\cdots P_n^{m_n})^{v_R}$ $(m_i>0)$.
Therefore, $[X]=[J]=[P_1]^{m_1}[P_2]^{m_2}\cdots [P_n]^{m_n}\in \langle [P] \mid P\in \H_1^{\rm reg}(R) \rangle$.
\end{proof}

We then obtain the following characterization.

\begin{thm}\label{5.9}
Consider the following conditions.
\benu[$(1)$]
\item
$R=R^*$ and $\Cl^v(R)=0$.
\item
$R$ is Boole $v$-regular.
\item
For every $I\in\F^v(R)$, there exists $a\in I$ such that $I^2=aI$.
\item
For every $I\in\F^v(R)$, we have $I\cong \tr_R(I)$.
\item
$R=R^*$ and every $P\in \H_1^{\rm reg}(R)$ is principal.
\enenu

Then $(1)\Leftrightarrow (2)\Leftrightarrow (3)\Leftrightarrow (4) \Rightarrow (5)$ holds.
If $\bar{R}$ is a finitely generated $R$-module, then $(5)\Rightarrow (1)$ also holds.
\end{thm}

\begin{proof}
$(4)\Rightarrow (3):$ This follows from the same argument as in the proof of $(3)\Rightarrow (2)$ in Theorem~\ref{3.4}.

$(3)\Rightarrow (2):$ This follows immediately from the definition of Boole $v$-regularity.

$(2)\Leftrightarrow (1):$ This follows from Theorem~\ref{5.5} and Lemma~\ref{5.11a} (1).

$(2)\Rightarrow (4):$ For $I\in \F^v(R)$, set $A=I:I$.
By Theorem~\ref{5.5}, we have $I, A\in \Inv^v(A)$, and Boole $v$-regularity gives $I\cong A$.
On the other hand, set $J=R:A$. Then $J$ is a trace ideal satisfying $J\in \F^v(R)$ and $J:J=R:J=R:(R:A)=A$.
The same argument gives $J\cong A$, and hence $I\cong J$.
Since the trace is invariant under isomorphism, we obtain $I\cong \tr_R(I)$.

$(1)\Rightarrow (5):$ Since $\H^{{\rm reg}}_1(R)\subseteq \Inv^v(R)$ and $\Cl^v(R)=0$, we have $P\cong R$ for every $P\in \H^{{\rm reg}}_1(R)$.

$(5)\Rightarrow (1):$ This follows from Proposition~\ref{5.8}.
\end{proof}

\begin{rem}
\benu[(1)]
\item
The equivalence of $(2)$ and $(3)$ in Theorem~\ref{5.9} also follows from \cite[Theorem 2.6]{KMN}.
In our setting, $t=v$; unless $R$ is a field, maximal $t$-ideals have height one, so the $t$-dimension is $1$.
\item
In connection with condition $(4)$ of Theorem~\ref{5.9}, the Lindo--Pande condition, which requires every ideal, not only every divisorial ideal, to be isomorphic to a trace ideal, is studied in \cite{KT}.
In this paper, we restrict attention to divisorial ideals.
\item
Arf domains are Boole $v$-regular by Theorem~\ref{3.4}.
\item
If $R$ is Boole $v$-regular, then preservation of strict closedness under localization and Lemma~\ref{5.11a} $(2)$ show that $R_Q$ is Boole $v$-regular for every $Q\in\Spec R$.
The converse does not hold.
Indeed, let $R$ be a Dedekind domain that is not a {\rm UFD}.
For every $Q\in\Spec R$, the ring $R_Q$ is a {\rm UFD}, and hence is Boole $v$-regular.
However, $\Cl^v(R)\neq 0$, so $R$ is not Boole $v$-regular.
\item
If $R$ is Boole $v$-regular, then $\calS^v(R)=\{[A]\mid A\in\B^v(R)\}$, which can be naturally identified with $\B^v(R)$.
In Proposition~\ref{5.15}, we show that in this case the $v$-product on $\calS^v(R)$ agrees with ordinary multiplication.

\enenu
\end{rem}

\begin{cor}\label{5.11}
Suppose that $R=R^*$. Let $A\in \B^v(R)$ and assume that $A$ is Boole $v$-regular.
Then, for every $B\in \B^v(R)$ satisfying $A\subseteq B$, $B$ is also Boole $v$-regular.
\end{cor}

\begin{proof}
Proposition~\ref{5.10} gives $B\in \B^v(A)$, so Theorem~\ref{3.12} yields $B=B^*$.
Applying Lemma~\ref{5.11a} $(1)$ to $A\subseteq B$ gives $0=\Cl^v(A)\twoheadrightarrow\Cl^v(B)$, and hence $\Cl^v(B)=0$.
By Theorem~\ref{5.9}, the ring $B$ is Boole $v$-regular.
\end{proof}

\begin{cor}\label{5.12}
Suppose that $\bar{R}$ is a finitely generated $R$-module.
Then, the following conditions are equivalent.
\benu[$(1)$]
\item
There exists $A\in \B^v(R)$ such that $A$ is Boole $v$-regular.
\item
$\bar{R}$ is a {\rm UFD}.

\enenu
\end{cor}

\begin{proof}
$(1)\Rightarrow (2):$
Since $A$ is Boole $v$-regular, we have $A=A^*$.
Also, $\bar A=\bar R$ and $\bar R\in\B^v(A)$, so Corollary~\ref{5.11} applied over $A$ shows that $\bar{R}$ is Boole $v$-regular.
As $\bar{R}$ is a Noetherian integrally closed domain and $\Cl^v(\bar{R})=0$, the ring $\bar{R}$ is a UFD (\cite[Proposition 2.2 (2)]{KM}).

$(2)\Rightarrow (1):$ We can take $A=\bar{R}$.
\end{proof}

\begin{lem}\label{6.4}
Suppose that $\Cl^v(R)=0$. If $I, J\in \F^v(R)$ satisfy ${\Ass_R R/I} \cap  {\Ass_R R/J}=\emptyset$, then $I\cap J=IJ$.

\end{lem}

\begin{proof}
It suffices to consider the case $I\ne R$ and $J\ne R$.
Let $0\ne x\in I\cap J$ and set
\[
H=\bigcap_{P\in\Ass_R R/I}(xR_P\cap R).
\]
Each $xR_P\cap R$ is $P$-primary and is divisorial, since
\[
(xR_P\cap R)^{v_R}\subseteq
\bigl((xR_P\cap R)^{v_R}\bigr)_P\cap R
=(xR_P)^{v_{R_P}}\cap R=xR_P\cap R.
\]
Thus, $H$ is divisorial.
For each $Q\in\H_1(R)$, we have $H_Q=xR_Q$ if $Q\in \Ass_R R/I$, and $H_Q=R_Q$ if $Q\notin \Ass_R R/I$.
By Lemma~\ref{5.a}, we have $H\in\Inv^v(R)$.
Since $\Cl^v(R)=0$, we can write $H=aR$.
The inclusions $xR\subseteq H\subseteq I$ imply $a\in I$, and there exists $b\in R$ such that $x=ab$.
By assumption, $Q\in\Ass_R R/J $ implies $Q\notin \Ass_R R/I$, so $aR_Q=H_Q=R_Q$ and $a\notin Q$.
Hence, $a$ is a non-zerodivisor on $R/J$, and $ab=x\in J$ gives $b\in J$.
Therefore, $x\in IJ$. Since the reverse inclusion is clear, this completes the proof.
\end{proof}

\begin{prop}\label{5.15}
If $R$ is Boole $v$-regular, then $XY\in \bsF^v(R)$ for every $X, Y\in \bsF^v(R)$.
Consequently, the $v$-products on $\bsF^v(R)$ and $\calS^v(R)$ agree with ordinary multiplication.
\end{prop}

\begin{proof}
Set $A=X:X$ and $B=Y:Y$.
By Theorem~\ref{5.9}, we can write $X=\alpha A$ and $Y=\beta B$, so it suffices to show that $AB$ is $R$-divisorial.
Set $C=A\cap B\in\B^v(R)$.
Theorem~\ref{3.12} and Lemma~\ref{5.11a} $(1)$ give $C=C^*$ and $\Cl^v(C)=0$, so $C$ is Boole $v$-regular.
By Proposition~\ref{5.10}, we have $A,B\in\B^v(C)$, and the $v_C$-closure and the $v_R$-closure agree on fractional ideals over $C$.
We may therefore replace the base ring by $C$ and assume that $R=A\cap B$.

Set $C_A=R:A$ and $C_B=R:B$.
If $P\in\Ass_R R/C_A\cap\Ass_R R/C_B $, then
\[
R:P\subseteq R:C_A=A,\qquad R:P\subseteq R:C_B=B
\]
and hence $R:P=R$.
However, since $P$ is divisorial, we have $P=R:(R:P)=R$, which is a contradiction.
Thus, $\Ass_R R/C_A \cap \Ass_R R/C_B=\emptyset$, and Lemma~\ref{6.4} gives $C_AC_B=C_A\cap C_B$.
Since $C_A:C_A=A$ and $C_B:C_B=B$, Boole $v$-regularity yields $C_A\cong_RA$ and $C_B\cong_RB$.
Therefore, $AB\cong C_A\cap C_B$, so $AB$ is divisorial.
\end{proof}

\begin{ex}\label{ex:boolean-chain}
Let $R_n=k[[t^2, t^{2n+1}]]$ be as in Example~\ref{ex:arf-chain}. Then $R_n$ is Boole $v$-regular, and
\[
 \calS^v(R_n)=\{[R_i] \mid 0\le i\le n\},\qquad
 [R_i][R_j]=[R_{\min\{i,j\}}].
\]
\end{ex}

\begin{ex}\label{ex:boole-surface}
Let $k=\mathbb Q$, $L=\mathbb Q(\sqrt2)$, $V=L[[s,t]]$, and $D=k[[s]]$, and set
\[
 R=D+tV
 \cong k[[s,u,v]]/(v^2-2u^2),
\]
where the isomorphism is given by $u\mapsto t$ and $v\mapsto\sqrt2t$.
Then $R$ is Boole $v$-regular, and
\[
\B^v(R)=\{R,V\},\qquad
 \Cl^v(R)=0,\qquad\calS^v(R)=\{[R],[V]\}.
\]
\end{ex}

\begin{proof}
Since $R=k[[s,t]]\oplus(\sqrt2t)k[[s,t]]$, we obtain the above hypersurface presentation.
Thus, $R$ is a two-dimensional Cohen-Macaulay local domain with integral closure $V=R[\sqrt2]$ and conductor $\mathfrak p=tV$.
By the same calculation as in Example~\ref{ex:pinch}, we obtain $R=R^*$ and $\B^v(R)=\{R,V\}$.

We use Proposition~\ref{5.8} to prove $\Cl^v(R)=0$.
As in Example~\ref{ex:pinch}, every $P\in\H_1^{\rm reg}(R)$ can be written as $P=P_f=fV\cap R$ for an irreducible element $f$ of $V$ satisfying $f(s,0)\ne0$.
Write $f(s,0)=s^mh(s)$ ($m\ge1$, $h(s)\in L[[s]]^\times$), and set $g=h(s)^{-1}f$.
Then $P=P_g=gV\cap R$, and $g(s,0)=s^m$ implies $g\in D+tV=R$.
If $a\in V$ satisfies $ga\in R$, then
\[
 s^ma(s,0)=(ga)(s,0)\in D=k[[s]].
\]
Since $a(s,0)\in L[[s]]$, we obtain $a(s,0)\in k[[s]]=D$ by comparing coefficients.
Hence, $a\in D+tV=R$, which yields $gV\cap R=gR$.
Thus, every element of $\H_1^{\rm reg}(R)$ is principal, and Proposition~\ref{5.8} gives $\Cl^v(R)=0$.

By Theorems~\ref{5.5} and \ref{5.9}, $R$ is Boole $v$-regular, and $\calS^v(R)=\{[R],[V]\}$.
\end{proof}

\section{Prime ideal decompositions of integrally closed ideals and local Boole \texorpdfstring{$v$}{v}-regularity}\label{sec:decomposition}

In this section, we construct a prime ideal decomposition of integrally closed divisorial ideals of strictly closed domains and determine when this decomposition holds under ordinary ideal multiplication.

We begin with the following lemma on the endomorphism rings of integrally closed ideals.

\begin{lem}[{\cite[Lemma 3.1, Proposition 3.2]{I1}}]\label{6.a}
Let $(0)\neq I$ be an integrally closed ideal of $R$. Then $I:I=(R:I)\cap \bar{R}$.
Consequently, if $I, J$ are integrally closed ideals of $R$ with $(0)\ne I\subseteq J$, then $I:I\supseteq J:J$.
\end{lem}

\begin{proof}
The inclusion $I:I\subseteq (R:I)\cap \bar R$ is clear.
Let $z\in(R:I)\cap\bar R$ and take $0\ne y\in I$.
Since $yz\in y\bar{R} \cap R=\bar{yR}\subseteq \bar I=I$, we have $zI\subseteq I$, that is, $z\in I:I$.
The last assertion follows from $R:I\supseteq R:J$.
\end{proof}

Throughout the remainder of this section, we assume that $R=R^*$.
For an integrally closed ideal $I\in\F^v(R)$, we set $R_{(I,0)}=R$ and $I_0=I$.
Suppose that $R_{(I,i)}$ and an ideal $I_i$ of $R_{(I,i)}$ have been constructed.
We take the radical $\sqrt{I_i}$ in $R_{(I,i)}$ and set
\[
 R_{(I,i+1)}=\sqrt{I_i}:\sqrt{I_i},\qquad
 I_{i+1}=I_i:\sqrt{I_i}\qquad(i\ge0).
\]
Since $I_{i+1}\subseteq R_{(I,i+1)}$ and $I_{i+1}R_{(I, i+1)}=I_{i+1}$, $I_{i+1}$ is an ideal of $R_{(I,i+1)}$. Thus, the construction can be continued.
In particular, each $I_i$ is an ideal of $R_{(I,i)}$, and
\[
 R=R_{(I,0)}\subseteq R_{(I,1)}\subseteq\cdots\subseteq R_{(I,i)}\subseteq\cdots\subseteq\bar R.
\]
We now prove the following properties of these rings and ideals.

We fix an integrally closed ideal $I\in \F^v(R)$, and set $R_i=R_{(I,i)}$ and $A=I:I$.
The above construction, together with Lemma~\ref{6.2} and Proposition~\ref{6.b} below, extends the construction for one-dimensional Arf rings in \cite[Definition 4.1, Lemma 4.3, Theorem 4.7]{I1} to arbitrary dimension.

\begin{lem}\label{6.2}
For every $i\ge 1$, the following assertions hold.
\benu[$(1)$]
\item
$R_i=R_i^*$ and $R_{i}\subseteq I_{i-1}:I_{i-1}$.
\item
$I_{i}\in \F^v(R_{i})$, and $I_i$ is an integrally closed ideal of $R_{i}$.
\item
$I_{i-1}=(\sqrt{I_{i-1}}I_{i})^{v_R}$.
\item
$I_i:I_i=A$.
\enenu
Moreover, $R_{i}\in\B^v(R)$ and $I_i\in\bsF^v(R)$ for every $i\ge0$.

\end{lem}

\begin{proof}
If $I=R$, then $R_i=I_i=A=R$ for every $i\ge0$, and all assertions are clear.
Suppose that $I\ne R$. First, we prove the assertions for $i=1$.

$(1)$
By Corollary~\ref{2.8} (1), we have $\sqrt{I}\in \F^v(R)$, so $R_1\in \B^v(R)$ and $I_1\in \bsF^v(R)$.
Hence, Theorem~\ref{3.12} gives $R_1=R_1^*$.
Moreover, Lemma~\ref{6.a} yields $R_1=\sqrt{I}:\sqrt{I} \subseteq I:I$.

$(2)$ We have already shown that $I_1$ is an ideal of $R_1$.
We show that it is integrally closed.
Let $z\in\overline{I_1}\subseteq R_1$ and write
\[
 z^m+b_1z^{m-1}+\cdots+b_m=0,
\]
where $b_j\in (I_1)^j$. For every $q\in \sqrt{I}$, we have $qz\in \sqrt{I}\subseteq R$ and
\[
 (qz)^m+(qb_1)(qz)^{m-1}+\cdots+q^mb_m=0,
 \quad q^jb_j\in(qI_1)^j\subseteq I^j.
\]
It follows that $qz\in \bar{I}=I$, so $z\sqrt{I} \subseteq I$, that is, $z\in I_1$.
Thus, $I_1$ is an integrally closed ideal of $R_1$.
Moreover, Proposition~\ref{5.10} gives $I_1^{v_{R_1}}=I_1^{v_R}=I_1$, and hence $I_1\in \F^v(R_1)$.
 
$(3)$
Similarly, $IR_1\subseteq IA=I$ implies $I\in \F(R_1)$, and hence $I\in \F^v(R_1)$.
Applying Theorem~\ref{5.5} (3) to $\sqrt{I}$, we obtain $\sqrt{I} \in\Inv^v(R_1)$.
Thus, $ (\sqrt{I}(R_1: \sqrt{I}))^{v_{R_1}}=R_1$, and
$$I=(IR_1)^{v_{R_1}}=(I(\sqrt{I}(R_1: \sqrt{I}))^{v_{R_1}})^{v_{R_1}}=(\sqrt{I}I(R_1: \sqrt{I}))^{v_{R_1}}\subseteq (\sqrt{I}I_1)^{v_{R_1}}\subseteq I^{v_{R_1}}=I.$$
Therefore, Proposition~\ref{5.10} yields $ I=(\sqrt{I}I_1)^{v_{R_1}}=(\sqrt{I}I_1)^{v_{R}}$.

$(4)$
If $u\in A$, then $u\sqrt{I}I_1\subseteq uI\subseteq I$ implies $uI_1\subseteq I_1$, so $u\in I_1:I_1$.
Conversely, if $0\neq u\in I_1:I_1$, then $uI=(u\sqrt{I}I_1)^{v_R}\subseteq(\sqrt{I}I_1)^{v_R}=I$ gives $u\in A$.
Hence, $I_1:I_1=A$.

Applying the above argument inductively, with the case $I_{i-1}=R_{i-1}$ being immediate, we obtain the assertions.
When $R_{i-1}\in\B^v(R)$, Remark~\ref{overring-reflexive} shows that both the ring $R_i\in\B^v(R_{i-1})$ and the $R_{i-1}$-divisorial ideal $I_i$ are $R$-divisorial. This completes the proof.
\end{proof}

\begin{prop}[{cf.\ \cite[Theorem 4.7, Corollary 4.8]{I1}}]\label{6.b}
The following assertions hold.
\benu[$(1)$]
\item
There exists $n\ge 0$ such that $I_{\ell}=R_{\ell}=A$ for every $\ell \ge n$.
\item
There exists $n\ge 0$ such that $I=(\sqrt{I}\sqrt{I_1}\cdots \sqrt{I_n})^{v_R}$.
\enenu

\end{prop}

\begin{proof}
$(1)$ For each $i$, we have $\sqrt{I_i}\subseteq R_i$ and $I_iR_i=I_i$, so
\[
 I_i\subseteq I_i:\sqrt{I_i}=I_{i+1}\subseteq R_{i+1}\subseteq A.
\]
We consider the ascending chain
\[
 I_0\subseteq I_1\subseteq I_2\subseteq\cdots\subseteq A
\]
of submodules of the finitely generated $R$-module $A$.
There exists $n\geq0$ such that $I_n=I_{n+1}$.
If $I_n\neq R_n$, choose the least $r>0$ such that $(\sqrt{I_n})^r\subseteq I_n$, and take $b\in (\sqrt{I_n})^{r-1}\setminus I_n$.
Then
\[
 b\sqrt{I_n}\subseteq (\sqrt{I_n})^r\subseteq I_n
\]
implies $b\in I_n:\sqrt{I_n}=I_{n+1}=I_n$, which is a contradiction.
Hence, $I_n=R_n$, and Lemma~\ref{6.2} (4) gives
\[
 A=I_n:I_n=R_n:R_n=R_n=I_n.
\]
For every $\ell\geq n$, we also have $I_n\subseteq I_{\ell} \subseteq R_{\ell}\subseteq A=I_n$, and therefore $I_{\ell}=R_{\ell}=A$.

$(2)$ By applying the equality in Lemma~\ref{6.2} $(3)$ repeatedly, we obtain
\[
 I=(\sqrt{I}\sqrt{I_1}\cdots \sqrt{I_{n}}I_{n+1})^{v_R}
\]
for every $n\ge 0$.
Choose $n$ such that $I_{n+1}=R_{n+1}=A$.
Since $\sqrt{I_n}I_{n+1}=\sqrt{I_n}R_{n+1}=\sqrt{I_n}$, we have
\[
 I =\bigl(\sqrt I\sqrt{I_1}\cdots\sqrt{I_n}\bigr)^{v_R},
\]
as desired.
\end{proof}

For an integrally closed ideal $I\in \F^v(R)$, we set $\rmA_i(I)=\Ass_{R_{i}} (R_{i}/I_i)$ and define
$$\displaystyle P(I)=\prod_{\{i\ge 0\mid I_i\ne R_i\}}\left(\prod_{P\in \rmA_i(I)} P\right).$$
These are ordinary products in $K$, and only finitely many factors occur by Proposition~\ref{6.b}.
When $I=R$, we set $P(I)=R$.

\begin{prop}\label{6.3}
If $I\in \F^v(R)$ is integrally closed, then $I=(P(I))^{v_R}$.

\end{prop}

\begin{proof}
Assume that $I\ne R$, and let $n+1$ be the least index such that $I_{n+1}=R_{n+1}$.
For $0\le i\le n$, Corollary~\ref{2.8} and Proposition~\ref{5.10} give $\sqrt{I_i}=\left(\prod_{P\in \rmA_i(I)} P\right)^{v_{R_{i}}}=\left(\prod_{P\in \rmA_i(I)} P\right)^{v_R}$.
Hence, Proposition~\ref{6.b} $(2)$ yields
$$I=\left(\prod_{\{i\ge 0\mid I_i\ne R_i\}}\left(\prod_{P\in \rmA_i(I)} P\right)^{v_R}\right)^{v_R}=(P(I))^{v_R}.$$
\end{proof}

We then obtain the following theorem.

\begin{thm}\label{6.5}
Suppose that $R=R^*$ and consider the following conditions.
\benu[$(1)$]
\item
For every $P\in \Spec R$, $R_P$ is Boole $v$-regular.
\item
For all $X, Y\in \bsF^v(R)$, we have $XY\in \bsF^v(R)$.
\item
$\Inv^v(R)=\Inv(R)$.
\item
$\H_1^{{\rm reg}}(R)\subseteq \Inv(R)$.
\item
If $I\in \F^v(R)$ is integrally closed, then $I=P(I)$.

\enenu
Then $(1)\Leftrightarrow (2)\Leftrightarrow  (3)\Rightarrow (4), (5)$ holds.
In addition, if $\bar{R}$ is a finitely generated $R$-module, then $(5)\Rightarrow (4)\Rightarrow (3)$ also holds.

\end{thm}

\begin{proof}
$(3)\Rightarrow (1):$ For $P\in\Spec R$, the map $\Inv^v(R)\to\Inv^v(R_P)$ is surjective by Lemma~\ref{5.11a} (2).
Under condition (3), every ideal in the image is invertible and is therefore principal over the local ring $R_P$.
Hence, $\Cl^v(R_P)=0$.
Since $(R_P)^*=(R^*)_P=R_P$, Theorem~\ref{5.9} shows that $R_P$ is Boole $v$-regular.

$(1)\Rightarrow (2):$ Let $X, Y\in \bsF^v(R)$.
By Proposition~\ref{5.15}, we have $((XY)^{v_R})_P=((XY)_P)^{v_{R_P}}=(XY)_P$ for every $P\in \Spec R$.
Comparing localizations at all prime ideals, we obtain $(XY)^{v_R}=XY$.

$(2)\Rightarrow (3):$ Let $X\in \Inv^v(R)$.
Since $X, R:X\in \bsF^v(R)$, condition $(2)$ gives $X(R:X)=(X(R:X))^{v_R}=R$, and hence $X\in\Inv(R)$.

$(2)\Rightarrow (5):$ Each prime ideal $P$ occurring in $P(I)$ is a height-one prime ideal of the corresponding ring $R_i$ and is therefore $R_i$-divisorial.
Since $R_i\in\B^v(R)$, Remark~\ref{overring-reflexive} shows that $P$ is also $R$-divisorial.
By assumption (2), their finite product $P(I)$ is $R$-divisorial.
Proposition~\ref{6.3} now gives $I=(P(I))^{v_R}=P(I)$.

$(3)\Rightarrow (4):$ This follows from $\H_1^{\rm reg}(R)\subseteq \Inv^v(R)$.

For the remaining implications, we assume that $\bar{R}$ is a finitely generated $R$-module.

$(5)\Rightarrow (4):$ The set $\Sigma=\H_1(R)\setminus\H_1^{\rm reg}(R)$ is finite because $\bar R$ is a finitely generated $R$-module.
Fix $P\in\H_1^{\rm reg}(R)$. By prime avoidance, choose
\[
0\ne a\in P\setminus\bigcup_{Q\in\Sigma}Q
\]
and set $I=aR$.
For $Q\in\Sigma$, we have $I_Q=R_Q$; at every other height-one prime, $I_Q$ is an ideal of a DVR and is therefore integrally closed.
Hence, $I$ is integrally closed.
Since $I:I=R$, every ring $R_i$ associated with this $I$ equals $R$.
By assumption (5), $aR=I=P(I)$ is an ordinary product of prime ideals of $R$.
As $P\in\Ass_R R/aR$, the prime ideal $P$ occurs as a factor.
Writing $P(I)=PJ$, we obtain $P(a^{-1}J)=R$, so $P$ is invertible.

$(4)\Rightarrow (3):$ This follows from Proposition~\ref{5.8}.
\end{proof}

\begin{rem}
If every nonzero proper principal ideal of a domain $R$ is a finite product of prime ideals of $R$, then $R$ is called a {\it $\pi$-domain}.
When $R$ is integrally closed, $R$ is a Krull domain, and the equivalence of conditions $(1)$--$(4)$ agrees with the characterization of $\pi$-domains in \cite[Theorem 1]{A}.
Thus, the theorem extends this characterization of Krull $\pi$-domains to the present setting, which includes non-integrally closed rings.
\end{rem}

\begin{ex}\label{ex:decomp-surface}
Let $V=L[[s,t]]$ and $R=k[[s]]+tV$ be as in Example~\ref{ex:boole-surface}, and set $I=s^2t^3V\subseteq R$.
Then $R$ is Boole $v$-regular, and we have
\[
\begin{array}{c|c|c|c}
 i&R_{(I, i)}&I_i&\Ass_{R_{(I, i)}} R_{(I, i)}/I_i\\ \hline
 0&R&s^2t^3V&\{sR,tV\}\\
 1&V&st^2V&\{sV,tV\}\\
 2&V&tV&\{tV\}\\
 3&V&V&\emptyset
\end{array}
\]
Thus
\[
 P(I)=(sR)(tV)(sV)(tV)(tV)=s^2t^3V=I.
\]
\end{ex}

\begin{proof}
The ideal $sR$ is prime because $R/sR\cong k+tL[[t]]$ is a domain.
Similarly, $tV$ is prime since $R/tV=k[[s]]$, and both primes have height one.
As a principal ideal of $V$, the ideal $I$ is integrally closed; since $I\subseteq R$, it is also integrally closed over $R$.
Moreover, $V\in\B^v(R)$ implies that $I$ is $R$-divisorial.
We have
\[
 \sqrt I=sR\cap tV=stV,\qquad
 (stV):(stV)=V,\qquad I:(stV)=st^2V.
\]
By computing the colon ideals in the UFD $V$, we obtain the remaining entries of the table.
At each step, the associated prime ideals of the quotient are precisely the height-one prime ideals in the table. Their ordinary product gives $P(I)$ and hence the stated decomposition.
\end{proof}

\begin{ex}\label{ex:decomp-affine}
Let $V=k[t]$ and $R=k[t^2,t^3]$ be as in Example~\ref{ex:affine-cusp}, and set $I=t^4V\subseteq R$.
The ring $R$ is locally Boole $v$-regular but is not Boole $v$-regular, since $\Cl^v(R)\cong(k,+)\ne0$.
In this case,
\[
\begin{array}{c|c|c|c}
 i&R_{(I,i)}&I_i&\Ass_{R_{(I,i)}} R_{(I,i)}/I_i\\ \hline
 0&R&t^4V&\{t^2V\}\\
 1&V&t^2V&\{tV\}\\
 2&V&tV&\{tV\}\\
 3&V&V&\emptyset
\end{array}
\]
and
\[
 P(I)=(t^2V)(tV)(tV)=t^4V=I.
\]
\end{ex}

\begin{proof}
The ideal $I$ is integrally closed and divisorial, and we have 
$$\sqrt{I}=t^2V,\qquad (t^2V):(t^2V)=V,\qquad  t^4V:t^2V=t^2V.$$
The remaining entries of the table follow by the same calculation in the PID $V$.
\end{proof}

\section{Polynomial extensions and constructions by localization}\label{sec:polynomial}

In this section, we study Clifford $v$-regularity of the polynomial ring $R[X]$, the structure of $\calS^v({R[X]})$, and Boole $v$-regularity.

If $R$ is a Noetherian domain satisfying $(S_2)$, then $R[X]$ is also a Noetherian domain satisfying $(S_2)$ (\cite{SP2}).
Moreover, $\bar{R[X]}=\bar{R}[X]$, and \cite[Lemma 4.9]{C} gives $(R[X])^*=R^*[X]$.
Hence, we obtain the following theorem.

\begin{thm}\label{7.1}
The following conditions are equivalent.
\benu[$(1)$]
\item
$R[X]=(R[X])^*$.
\item
$R[X]$ is Clifford $v$-regular.
\item
$R=R^*$.
\item
$R$ is Clifford $v$-regular.

\enenu
\end{thm}

For $Y\in \bsF(R)$, we set $Y[X]=YR[X]$.
Then $Y[X]\in \bsF(R[X])$, and
$$R[X]:Y[X]=(R:Y)R[X]=(R:Y)[X].$$
In particular, $(Y[X])^{v_{R[X]}}=(Y^{v_R})[X]$, so the map $\Psi : \bsF^v(R)\to \bsF^v(R[X]),\ Y \mapsto Y[X]$ is well defined.
We then have the following result on the structure of $\calS^v({R[X]})$.

\begin{thm}\label{7.2}
If $R=R^*$, then the map
$$\psi: \B^v(R)\to \B^v({R[X]}),\ A\mapsto A[X]$$
is bijective. Consequently, we have
\[
\calS^v(R[X])=\bigsqcup_{A\in\B^v(R)}\calS_{A[X]}^v(R[X]) \quad \text{and}
\quad \calS_{A[X]}^v(R[X])\cong\Cl^v(A[X]).
\]

\end{thm}

\begin{proof}
Let $B\in \B^v(R[X])$ and set $A=B\cap K$.
Then $A$ is a ring and $A[X]\subseteq B$.
Since $R[X]\subseteq B\subseteq \bar{R[X]}=\bar{R}[X]\subseteq K[X]$, we have
$$R=R[X]\cap K\subseteq A \subseteq \bar{R}[X]\cap K=\bar{R}.$$
Moreover, $A$ is a finite birational extension of $R$.
Indeed, since $B$ is a finitely generated $R[X]$-module, we can choose $0\ne f\in R[X]$ such that $fB\subseteq R[X]$.
Let $0\ne c\in R$ be the leading coefficient of $f$.
For every $a\in A$, we have $af\in R[X]$, so $ca\in R$.
Thus, $cA\subseteq R$, and hence $A$ is a finitely generated $R$-module.

Furthermore, $(A^{v_R})[X]=(A[X])^{v_{R[X]}}\subseteq B^{v_{R[X]}}=B$ implies $A^{v_R}=(A^{v_R})[X]\cap K\subseteq B\cap K=A$, so $A\in \B^v(R)$ and $A[X]\in \B^v(R[X])$.
Theorem~\ref{3.12} also gives $A=A^*$.

Since $R[X]=(R[X])^*$ and $A[X]\in \B^v(R[X])$, Proposition~\ref{5.10} yields $B\in \B^v(A[X])$.
Thus, we have
$$A=A^*\in \B^v(R),\quad B\in \B^v(A[X]),\quad B\cap K=A.$$
We show that $A[X]=B$.
Suppose that $Q\in \H_1(A[X])$ satisfies $A[X]_Q\ne B_Q$. Then $A[X]:B\subseteq Q$.
Set $P=Q\cap A$.
If $P=(0)$, then $A[X]_Q=K[X]_{QK[X]}$ is integrally closed, so $A[X]_Q=B_Q$, which is a contradiction.
Hence, $(0)\ne P\in \H_1(A)$ and $Q=P[X]$.
It follows that
$$(P:P)[X]=Q:Q\subseteq A[X]:Q\subseteq A[X]:(A[X]:B)=B$$
and $A\subseteq P:P=(P:P)[X]\cap K\subseteq B\cap K=A$ gives $A=P:P$.
By Lemma~\ref{5.a}, the ideal $PA_P$ is principal.
Thus, $A_P$ is a DVR, and $A_P[X]$ is integrally closed.
Consequently,
\[
 A_P[X]=(A\setminus P)^{-1}A[X]=(A\setminus P)^{-1}B.
\]
Localizing further at $Q=P[X]$ gives $A[X]_Q=B_Q$, which is a contradiction.
Therefore, $A[X]_Q=B_Q$ for every $Q\in \H_1(A[X])$.
Since $B\in \B^v(A[X])$, we obtain $A[X]=B$, which proves surjectivity.
Injectivity follows from $A[X]\cap K=A$ for every $A\in\B^v(R)$.
\end{proof}

We next consider Boole $v$-regularity of $R[X]$.
If $I\in \Inv^v(R)$, then $I[X]\in \Inv^v(R[X])$, so $\iota : \Cl^v(R)\to \Cl^v(R[X]), \ [I]\mapsto [I[X]]$ is an injective group homomorphism.
Indeed, if $I[X]=fR[X]$ for some $f\in K(X)^\times$, extending to
$K[X]$ yields $K[X]=fK[X]$, so $f\in K^\times$.
Comparing constant coefficients, we obtain $I=fR$, which proves injectivity.

The following result is stated explicitly in \cite[Corollary 1.9]{EIK}, where \cite[Theorem 3.6]{Ga} is cited as the original source.
We prove only the necessity, which we use below.
For the sufficiency, we refer the reader to these papers.

\begin{prop}[{\cite[Corollary 1.9]{EIK}}]\label{7.3}
The map $\iota :\Cl^v(R)\to \Cl^v(R[X])$ is bijective if and only if $R$ is integrally closed.

\end{prop}

\begin{proof}
(Necessity) Set $S=R[X]$ and suppose that $\iota$ is surjective.
For any $\alpha\in\bar R$, set
\[
Q_\alpha=(X-\alpha)K[X]\cap S
=\Ker(S\longrightarrow R[\alpha],\ X\longmapsto\alpha).
\]
Then $Q_\alpha\cap R=(0)$, and $S_{Q_\alpha}=K[X]_{(X-\alpha)}$ is a DVR.
Hence, Proposition~\ref{2.4} and Lemma~\ref{5.a} give $Q_\alpha\in\Inv^v(S)$.
By surjectivity, there exist $I\in\Inv^v(R)$ and $f\in K(X)^\times$ such that $Q_\alpha=fI[X]$.
Extending both sides to $K[X]$, we obtain
\[
(X-\alpha)K[X]=fK[X].
\]
Thus, we may write $f=c(X-\alpha)$ with $c\in K^\times$.
Setting $J=cI$, we have $Q_\alpha=(X-\alpha)J[X]$.
For every $a\in J$, we have $a(X-\alpha)\in S$, so by comparing coefficients we obtain $a\in R$.
Thus, $J\subseteq R$.

Since $\alpha$ is integral over $R$, the ideal $Q_\alpha$ contains a monic polynomial $F$.
Write $F=(X-\alpha)G$ with $G\in J[X]$.
The leading coefficient of $G$ is 1, so $1\in J$.
Hence, $J=R$, and $X-\alpha\in Q_\alpha\subseteq S$ implies $\alpha\in R$.
Therefore, $R$ is integrally closed. This proves the necessity.
\end{proof}

We then obtain the following.

\begin{thm}\label{7.4}
The following conditions are equivalent.
\benu[$(1)$]
\item
$R=R^*$ and $\Cl^v(R[X])=0$.
\item
$R[X]$ is Boole $v$-regular.
\item
$R$ is a {\rm UFD}.

\enenu
\end{thm}

\begin{proof}
$(1)\Leftrightarrow (2):$ This follows from Theorems~\ref{7.1} and \ref{5.9}.

$(3)\Rightarrow (2):$ The ring $R[X]$ is also a {\rm UFD}, so every divisorial ideal is principal.
Hence, $R[X]$ is Boole $v$-regular.

$(1)\Rightarrow (3):$ Since $\iota$ is injective, we have $\Cl^v(R)=0$.
Thus, $\iota$ is an isomorphism between trivial groups, and Proposition~\ref{7.3} shows that $R$ is integrally closed.
Every height-one prime ideal of a Noetherian integrally closed domain is $v$-invertible, so the vanishing of the class group implies that all such ideals are principal.
Hence, $R$ is a UFD.
\end{proof}


We next consider local Boole $v$-regularity of $R[X]$.
We begin with the following lemma.

\begin{lem}[{cf.\ \cite[Theorem 2]{NN}}]\label{lem:monic-invertible}
Let $D$ be a Noetherian domain, let $F=Q(D)$ be the field of fractions of $D$, and let $\alpha\in \bar{D}$.
If
$$Q=\Ker(D[X]\to D[\alpha],\ X\mapsto\alpha)$$
is invertible over $D[X]$, then $\alpha\in D$.
\end{lem}

\begin{proof}
By localizing at maximal ideals, we may assume that $(D,\mathfrak m,k)$ is local.
Set $S=D[X]$ and $E=D[\alpha]$.
Since $\alpha$ is integral over $D$, the ideal $Q$ contains a monic polynomial, so its image $\overline Q$ in $k[X]$ is nonzero.
Also, $E/\mathfrak mE\ne0$ by Nakayama's lemma, since $E\ne0$ is a finitely generated $D$-module.
Thus, $\overline Q$ is proper, and we may write $\overline Q=h k[X]$ with $h$ monic and $n=\deg h\ge1$.
Since $E$ is a finitely generated $D$-module and $E/\m E\cong k[X]/(h)$, Nakayama's lemma shows that $1,\alpha,\ldots,\alpha^{n-1}$ generate $E$.
Thus, we can choose a monic polynomial $f\in S$ of degree $n$ satisfying $f(\alpha)=0$.
Then $f\in Q$ and $\overline f=h$.

Set $H=f(S:Q)$.
Since $Q$ is invertible, we have
\[
 fS\subseteq H\subseteq S,\qquad QH=fS.
\]
By taking images in $k[X]$ and using $\overline Q=\overline f k[X]$ and $\overline f\ne0$, we obtain $\overline H=k[X]$, and hence $H+\mathfrak mS=S$.
On the other hand, since $f$ is monic, $S/H$ is a finitely generated $D$-module.
Nakayama's lemma therefore gives $H=S$, so $Q=fS$.
Extending to $F[X]$, we obtain $fF[X]=(X-\alpha)F[X]$.
Since both generators are monic, $f=X-\alpha$, and hence $\alpha\in D$, as desired.
\end{proof}

\begin{thm}\label{7.5}
The following conditions are equivalent.
\benu[$(1)$]
\item
For every $Q\in \Spec R[X]$, $R[X]_Q$ is Boole $v$-regular.
\item
For every $Q\in \Spec R[X]$, $R[X]_Q$ is a {\rm UFD}.
\item
For every $P\in \Spec R$, $R_P$ is a {\rm UFD}.
\enenu

\end{thm}

\begin{proof}
$(3)\Rightarrow (2):$ Let $Q\in \Spec R[X]$ and set $P=Q\cap R$.
Then $R[X]_Q=(R_P[X])_{QR_P[X]}$.
Since $R_P$ is a UFD, so is $R[X]_Q$.

$(2)\Rightarrow (1):$ Every UFD is Boole $v$-regular.

$(1)\Rightarrow(3):$
Set $S=R[X]$.
By assumption, every $J\in\Inv^v(S)$ becomes principal in each localization $S_Q$.
Hence, $J\in \Inv(S)$, that is, $\Inv^v(S)=\Inv(S)$.
For any $\alpha\in\bar R$, set
\[
 Q_\alpha=(X-\alpha)K[X]\cap S
 =\Ker(S\to R[\alpha]).
\]
Then $S_{Q_\alpha}=K[X]_{(X-\alpha)}$ is a DVR, and $(Q_\alpha)_Q=S_Q$ at every other height-one prime.
Thus, Lemma~\ref{5.a} gives $Q_\alpha\in\Inv^v(S)=\Inv(S)$.
Lemma~\ref{lem:monic-invertible} now implies $\alpha\in R$, so $R$ is integrally closed.

Next, take $P\in\H_1(R)$.
Since $R$ is integrally closed, we have $P\in\Inv^v(R)$, and hence $P[X]\in\Inv^v(S)=\Inv(S)$.
Therefore
\[
 S=P[X](S:P[X])=\bigl(P(R:P)\bigr)[X]
\]
and comparison of constant terms gives $P(R:P)=R$.
Thus, every height-one prime ideal of $R$ is invertible.
Consequently, for every $P\in\Spec R$, all height-one prime ideals of $R_P$ are principal, so $R_P$ is a UFD. This proves the remaining implication.
\end{proof}

By Theorems~\ref{7.4} and \ref{7.5}, polynomial extension alone does not produce non-integrally closed Boole $v$-regular domains or non-integrally closed locally Boole $v$-regular domains.
In contrast, by taking suitable localizations of polynomial extensions, we can construct examples of both types in every positive dimension.

Let $d\ge1$ be an integer and let $k$ be a field. We set
\[
 D=k[x_1,\ldots,x_{d-1}],\qquad A=D[t^2,t^3],\qquad B=D[t].
\]
When $d=1$, we interpret $D=k$.
The isomorphism
\[
 A\cong k[x_1,\ldots,x_{d-1},u,v]/(v^2-u^3)
\]
given by $u\mapsto t^2$ and $v\mapsto t^3$ shows that $A$ is a $d$-dimensional Cohen-Macaulay domain.
Moreover, $B=A+At$ is the integral closure of $A$ and is finite as an $A$-module.
Since $A=D\oplus t^2D[t]$, the conductor is
\[
 \mathfrak p=A:B=t^2B=(t^2,t^3)A.
\]
Thus, for any multiplicatively closed set $W\subseteq A\setminus\{0\}$, the ring $R=W^{-1}A$ is a Cohen-Macaulay domain, $T=W^{-1}B$ is its integral closure, and $\fkp R=R:T=t^2T$.
In particular, $T\in\B^v(R)$.
If $W\cap\fkp\ne\emptyset$, then $R=T$.
If $W\cap\fkp=\emptyset$, then $\fkp R\in\H_1(R)$, and for $R\subsetneq H\in\B^v(R)$ we have $\fkp R\subseteq R:H\subsetneq R$.
As in Example~\ref{ex:pinch}, the divisoriality of $R:H$ gives $R:H=\fkp R$, and hence $H=T$.
Therefore, in either case,
\[
 \B^v(R)=\{R,T\},
\]
where the right-hand side is a singleton when $R=T$.
On the other hand, since $A=k[t^2,t^3][x_1,\ldots,x_{d-1}]$, Example~\ref{ex:affine-cusp} and Theorem~\ref{7.1} give $A=A^*$.
By the compatibility with localization discussed after Definition~\ref{3.7}, we obtain $R=R^*$.

\Needspace{12\baselineskip}
\begin{ex}\label{ex:boole-anydim}
With the above notation, let
\[
 \mathfrak m=(t^2-1,t^3-1,x_1,\ldots,x_{d-1})A,
 \qquad W_1=A\setminus(\mathfrak p\cup\mathfrak m)
\]
and set $R_1=W_1^{-1}A$ and $T_1=W_1^{-1}B$.
Then $R_1$ is a non-integrally closed $d$-dimensional Cohen-Macaulay semi-local domain that is Boole $v$-regular, and
\[
 \B^v(R_1)=\{R_1,T_1\},\qquad \Cl^v(R_1)=0,\qquad
 \calS^v(R_1)=\{[R_1],[T_1]\}.
\]
\end{ex}
\begin{proof}
The prime ideal $\mathfrak p$ has height one, and $\mathfrak m$ is a maximal ideal of height $d$; neither contains the other.
Hence, we have $\Max R_1=\{\mathfrak pR_1,\mathfrak mR_1\}$.
Since $\dim(R_1)_{\mathfrak mR_1}=\dim A_{\mathfrak m}=d$, we have $\dim R_1=d$.
Moreover, $R_1:T_1=\fkp R_1\nsubseteq\mathfrak mR_1$ implies that
\[
 (R_1)_{\mathfrak mR_1}=(T_1)_{\mathfrak mR_1}
 =D[t]_{(t-1,x_1,\ldots,x_{d-1})}
\]
is a regular local UFD.
On the other hand, the above construction gives $R_1=R_1^*$.
Since $\dim(R_1)_{\fkp R_1}=1$, Theorems~\ref{3.9} and \ref{3.4} show that $(R_1)_{\fkp R_1}$ is Arf and Boole $v$-regular.
Furthermore, setting $F=Q(D)$, we have
\[
 (R_1)_{\fkp R_1}=A_\fkp=F[t^2,t^3]_{(t^2,t^3)}
\]
and hence $t\notin(R_1)_{\fkp R_1}$.
Thus, $t\notin R_1$, so $R_1\subsetneq T_1$.

Take $I\in\Inv^v(R_1)$.
Both localizations at maximal ideals are Boole $v$-regular, so their $v$-class groups are trivial.
Hence, $I$ becomes principal at each maximal ideal.
Therefore, $I$ is invertible over $R_1$.
Since $R_1$ is semi-local, every invertible fractional ideal is principal, and we obtain $I\cong R_1$.
Thus, $\Cl^v(R_1)=0$, and Theorem~\ref{5.9} shows that $R_1$ is Boole $v$-regular.
The description of the class semigroup follows from $\B^v(R_1)=\{R_1,T_1\}$, obtained in the above construction, and Theorem~\ref{5.5}.
\end{proof}

\begin{ex}\label{ex:locally-boole-anydim}
With the above notation, suppose that $\operatorname{char}k\ne2$.
From each associate class of nonunit irreducible elements of $D=k[x_1,\ldots,x_{d-1}]$, choose a representative $f$ satisfying $f(0,\ldots,0)\in\{0,1\}$, and denote the set of these representatives by $\mathcal P$.
Let $W_2$ be the multiplicatively closed set generated by $g_f=t^2+f$ $(f\in\mathcal P)$, and set
\[
 R_2=W_2^{-1}A,\qquad T_2=W_2^{-1}B.
\]
When $d=1$, we have $D=k$, $\mathcal P=\emptyset$, and $W_2=\{1\}$, so $R_2=k[t^2,t^3]$.
Then $R_2$ is a non-integrally closed $d$-dimensional Cohen-Macaulay domain such that $(R_2)_P$ is Boole $v$-regular for every $P\in\Spec R_2$, whereas $R_2$ itself is not Boole $v$-regular.
Moreover, if we set $F=Q(D)$, then
\[
 \Cl^v(R_2)\cong(F,+).
\]
More precisely, for $\lambda=\frac{a}{b}\in F$ $(a,b\in D,\ b\ne0)$, we set
\[
 J_\lambda=(b+at)R_2+\fkp R_2.
\]
This is an invertible fractional ideal independent of the choice of representation of $\lambda$, and $\lambda\mapsto[J_\lambda]$ gives the above isomorphism.
In particular,
\[
 I=(t^2+t^3,t^4,t^5)R_2=t^2J_1
\]
is a nonprincipal invertible ideal whose class $[I]=[J_1]$ corresponds to $1\in F$.
\end{ex}
\begin{proof}
Let $\mathfrak m$ be the maximal ideal used in Example~\ref{ex:boole-anydim}.
The image of each $g_f$ modulo $\mathfrak m$ is $1+f(0,\ldots,0)\in\{1,2\}$, which is nonzero since $\operatorname{char}k\ne2$.
Thus, $W_2\cap\mathfrak m=\emptyset$, and $(R_2)_{\mathfrak mR_2}=A_{\mathfrak m}$ gives $\dim R_2=d$.
The conductor is $\fkp R_2=R_2:T_2=t^2T_2=(t^2,t^3)R_2$.
Here, the image of each $g_f$ modulo $\fkp R_2$ is $f$.
Since $D$ is a UFD, the image of every nonzero element of $D$ is invertible in $R_2/\fkp R_2$.
Hence
\[
 R_2/\fkp R_2=F
\]
and $\fkp R_2$ is a maximal ideal.
Since $(\fkp R_2)\cap A=\fkp$, we have
\[
 (R_2)_{\fkp R_2}=A_\fkp=F[t^2,t^3]_{(t^2,t^3)}.
\]
This is the one-dimensional Arf local ring appearing in Example~\ref{ex:boole-anydim}, and hence is Boole $v$-regular.
In particular, $\fkp R_2$ has height one.
Since $t\notin(R_2)_{\fkp R_2}$, we have $t\notin R_2$, so $R_2$ is not integrally closed.
If $P\in\Spec R_2\setminus\{\fkp R_2\}$, then $\fkp R_2\nsubseteq P$, and localizing the conductor gives $(R_2)_P=(T_2)_P$.
Since the right-hand side is a localization of the UFD $T_2$, it is also a UFD and hence is Boole $v$-regular.
Thus, $R_2$ is locally Boole $v$-regular.

We now compute the class group.
Each $g_f=t^2+f$ is irreducible in the UFD $B=D[t]$ by Eisenstein's criterion with respect to the prime element $f$ of $D$.
Hence, $T_2^\times$ is generated by $k^\times$ and $g_f^{\pm1}$ $(f\in\mathcal P)$.
All these generators are also units of $R_2$, so
\[
 T_2^\times=R_2^\times.
\]
Moreover, localizing $B/t^2B=D[t]/(t^2)$ gives
\[
 T_2/\fkp R_2=F[\varepsilon]/(\varepsilon^2),
\]
where $\varepsilon$ denotes the image of $t$ in $T_2/\fkp R_2$.
Let $\pi:T_2\longrightarrow F[\varepsilon]/(\varepsilon^2)$ denote the natural surjection.
Since $R_2/\fkp R_2=F$ and $\Ker\pi=\fkp R_2\subseteq R_2$, we have $\pi^{-1}(F)=R_2$.

Take $\lambda=\frac{a}{b}\in F$ $(a,b\in D,\ b\ne0)$.
Since $\pi(b+at)=b(1+\lambda\varepsilon)$, we have
\[
 J_\lambda=(b+at)R_2+\fkp R_2
 =\pi^{-1}\bigl(F(1+\lambda\varepsilon)\bigr).
\]
Thus, $J_\lambda$ is independent of the representation of $\lambda$, and $J_0=R_2$.
The module $J_\lambda$ is a nonzero submodule of the finitely generated $R_2$-module $T_2$ and satisfies $t^2J_\lambda\subseteq\fkp R_2\subseteq R_2$, so it is a nonzero fractional ideal.
Since $\pi(b+at)$ is a unit of $T_2/\fkp R_2$ and $J_\lambda$ contains $\fkp R_2$, we have
\[
 J_\lambda T_2=(b+at)T_2+\fkp R_2=T_2.
\]
Consequently, $J_\lambda(\fkp R_2)=t^2J_\lambda T_2=\fkp R_2$.
Since $\fkp R_2\subseteq J_\mu$, we obtain, for any $\lambda,\mu\in F$,
\[
 \fkp R_2=J_\lambda(\fkp R_2)\subseteq J_\lambda J_\mu.
\]
Furthermore, $(1+\lambda\varepsilon)(1+\mu\varepsilon)=1+(\lambda+\mu)\varepsilon$ in $T_2/\fkp R_2$, so
\[
 J_\lambda J_\mu=J_{\lambda+\mu}.
\]
Thus, $J_\lambda J_{-\lambda}=J_0=R_2$ shows that $J_\lambda$ is invertible, and the map
$ \Phi:(F,+)\longrightarrow\Cl^v(R_2),
 \lambda\longmapsto[J_\lambda]$
is a group homomorphism.
If $[J_\lambda]=[J_0]=[R_2]$, then $J_\lambda=cR_2$ for some $c\in Q(R_2)^\times$.
Extending to $T_2$ gives $T_2=cT_2$, so $c\in T_2^\times=R_2^\times$.
Hence, $J_\lambda=R_2$, and therefore $F(1+\lambda\varepsilon)=F$.
By comparing the coefficients of $\varepsilon$, we obtain $\lambda=0$.
Thus, $\Phi$ is injective.

To prove surjectivity, we use Proposition~\ref{5.8}.
Let $P\in\H_1^{\rm reg}(R_2)$.
Since $\H_1(R_2)\setminus \H_1^{\rm reg}(R_2)=\{\fkp R_2\}$, we have $P\ne\fkp R_2$.
The rings $R_2$ and $T_2$ coincide after localization away from the conductor.
Since $T_2$ is a localization of the UFD $B$, we can choose an irreducible element $h$ of $B=D[t]$ such that $P=hT_2\cap R_2$.
As $h$ is not associated to $t$, we may write
\[
 h=h_0+h_1t+t^2h_2(t),\qquad
 h_0,h_1\in D,\quad h_0\ne0,\quad h_2(t)\in D[t].
\]
Set $\nu=\frac{h_1}{h_0}\in F$. Then $\pi(h)=h_0(1+\nu\varepsilon)$.
For any $z\in T_2$, the condition $hz\in R_2$ holds if and only if $h_0(1+\nu\varepsilon)\pi(z)=\pi(h)\pi(z)\in F$, which is equivalent to $\pi(z)\in F(1-\nu\varepsilon)$, that is, $z\in J_{-\nu}$.
Hence, $P=hJ_{-\nu}$, and we obtain $[P]=\Phi(-\nu)$.
By the above construction, we have $R_2=R_2^*$, and the integral closure $\overline{R_2}=T_2$ is a finitely generated $R_2$-module.
Therefore, Proposition~\ref{5.8} gives
\[
 \Cl^v(R_2)
 =\bigl\langle[P]\mid P\in\H_1^{\rm reg}(R_2)\bigr\rangle
 =\Im\Phi.
\]
Thus, $\Phi$ is also surjective, and we obtain $\Cl^v(R_2)\cong(F,+)$.
Since $\Cl^v(R_2)\ne0$, Theorem~\ref{5.9} shows that $R_2$ is not Boole $v$-regular.

Finally, $\fkp R_2=(t^2,t^3)R_2$ gives
\[
 t^2J_1=t^2\bigl((1+t)R_2+\fkp R_2\bigr)
 =(t^2+t^3,t^4,t^5)R_2=I.
\]
Hence, $I$ is invertible, and $[I]=[J_1]\neq [J_0]=[R_2]$ shows that it is not principal. This completes the proof.
\end{proof}

\section*{Declaration of generative AI and AI-assisted technologies
in the manuscript preparation process}

The author used ChatGPT (OpenAI) to assist with English
translation, language revision, manuscript organization, literature
searches, and the exploration of proof strategies and mathematical
examples.
After using this tool, the author reviewed and edited the content
as needed and takes full responsibility for the content of the
published article.

\end{document}